\documentclass[11pt, notitlepage]{article}
\usepackage{amssymb,amsmath,comment}
\catcode`\@=11 \@addtoreset{equation}{section}
\def\thesection{\arabic{section}}

\def\theequation{\thesection.\arabic{equation}}
\catcode`\@=12
\usepackage{colortbl}%
\usepackage{a4wide}

\newcommand{\fa} {\forall}
\newcommand{\ds} {\displaystyle}
\newcommand{\e}{\epsilon}

\newcommand{\al} {\alpha}
\newcommand{\ba} {\beta}
\newcommand{\de} {\delta}

\newcommand{\Om} {\Omega}
\newcommand{\ra} {\rightarrow}

\newcommand{\De} {\Delta}
\newcommand{\la} {\lambda}

\newcommand{\noi} {\noindent}

\newcommand{\oline} {\overline}
\newcommand{\mb} {\mathbb}
\newcommand{\mc} {\mathcal}
\newcommand{\lra} {\longrightarrow}
\newcommand{\ld} {\langle}
\newcommand{\rd} {\rangle}

\usepackage[all]{xy}
\catcode`\@=11
\def\theequation{\@arabic{\c@section}.\@arabic{\c@equation}}
\catcode`\@=12

\def\proof{\noindent{\textbf{Proof. }}}
\def\QED{\hfill {$\square$}\goodbreak \medskip}

\newtheorem{Theorem}{Theorem}[section]
\newtheorem{Lemma}[Theorem]{Lemma}

\newtheorem{Corollary}[Theorem]{Corollary}

\newtheorem{Definition}[Theorem]{Definition}

\begin{document}
{\vspace{0.01in}}

\title
{Existence of multiple solutions of $p$-fractional Laplace operator
with sign-changing weight function}

\author{
{\bf  Sarika Goyal\footnote{email: sarika1.iitd@gmail.com}} and {\bf  K. Sreenadh\footnote{e-mail: sreenadh@gmail.com}}\\
{\small Department of Mathematics}, \\{\small Indian Institute of Technology Delhi}\\
{\small Hauz Khaz}, {\small New Delhi-16, India.}\\
 }

\date{}

\maketitle

\begin{abstract}

In this article, we study the following $p$-fractional Laplacian equation
 \begin{equation*}
 (P_{\la}) \left\{
\begin{array}{lr}
- 2\int_{\mb R^n}\frac{|u(y)-u(x)|^{p-2}(u(y)-u(x))}{|x-y|^{n+p\al}}
dy =  \la |u(x)|^{p-2}u(x) + b(x)|u(x)|^{\ba-2}u(x)\; \text{in}\;
\Om \\
 \quad \quad\quad\quad \quad\quad\quad\quad\quad \quad u = 0 \; \mbox{in}\; \mb R^n \setminus\Om,\quad u\in W^{\al,p}(\mb R^n).\\
\end{array}
\quad \right.
\end{equation*}
where $\Om$ is a bounded domain in $\mb R^n$ with smooth boundary,
$n> p\al$, $p\geq 2$, $\al\in(0,1)$, $\la>0$ and $b:\Om\subset\mb R^n \ra \mb R$ is a
sign-changing continuous function. We show the existence and multiplicity of non-negative solutions of $(P_{\la})$ with respect to the parameter $\la$, which changes according to whether $1<\ba<p$ or $p< \ba< p^{*}=\frac{np}{n-p\al}$ respectively. We discuss both the cases separately. Non-existence results are also obtained.
\medskip

\noi \textbf{Key words:} Non-local operator, fractional Laplacian,
sign-changing weight function, Nehari manifold.

\medskip

\noi \textit{2010 Mathematics Subject Classification:} 35A15, 35B33,
35H39
\end{abstract}

\bigskip
\vfill\eject

\section{Introduction}
\noi The aim of this article is to study the existence and multiplicity of non-negative solutions of following equation which is driven by the
non-local operator $\mc L_{K}$ as
\begin{equation}\label{eq01}
\left\{
\begin{array}{lr}
 -\mc L_{K}(u) =  \la |u(x)|^{p-2}u(x) + b(x)|u(x)|^{\ba-2}u(x)\; \text{in}\;
\Om \\
\quad\quad u = 0 \;\text{on}\; \mb R^n \setminus \Om,
\end{array}
\right.
\end{equation}
\noi where $\mc L_{K}$ is defined as
\[\mc L_K u(x)= 2 \int_{\mb R^n}|u(y)- u(x)|^{p-2} (u(y)-u(x))K(x-y)
dy\;\;\text{for\; all}\;\; x\in \mb R^n,\]
and $K :\mb R^n\setminus\{0\}\ra(0,\infty)$ satisfying:\\
$(a)\; mK \in L^1(\mb R^n),\;\text{where}\; m(x) = \min\{1, |x|^p\}$,\\
$(b)$ there exist $\theta>0$ and $\al\in(0,1)$ such that $K(x)\geq
\theta |x|^{-(n+p\al)},$\\
$(c)\; K(x) = K(-x)$ for any $ x\in \mb R^n\setminus\{0\}$.\\
Here $\Om$ is a bounded domain in $\mb R^n$ with smooth boundary,
$n> p\al$, $p\geq 2$, $\al\in(0,1)$, $\la>0$ and $b:\Om \ra \mb R$ is a
sign-changing continuous function.\\
\noi In particular, if $K(x) = |x|^{-(n+p\al)}$ then $\mc L_{K}$ becomes $p$-fractional Laplacian operator and is denoted by $(-\De)^{\al}_{p}$.

\noi Recently a lot of attention is given to the study of fractional and
non-local operators of elliptic type due to concrete real world
applications in finance, thin obstacle problem, optimization,
quasi-geostrophic flow etc. Dirichlet boundary value problem in case
of fractional Laplacian with polynomial type nonlinearity using
variational methods is recently studied in
\cite{tan,mp,var,weak,ls,yu}. Also existence and multiplicity
results for non-local operators with convex-concave type
nonlinearity is shown in \cite{mul}. In case of square root of
Laplacian, existence and multiplicity results for sublinear and superlinear type of nonlinearity with sign-changing
weight function is studied in \cite{yu}. In \cite{yu}, author used
the idea of Caffarelli and Silvestre \cite{cs}, which gives a
formulation of the fractional Laplacian through Dirichlet-Neumann
maps. Recently eigenvalue problem related to $p-$fractional Laplacian is studied in \cite{gf,pt}.

\noi For $\al=1$, a lot of work has been done for
multiplicity of positive solutions of semilinear elliptic problems
with positive nonlinearities \cite{ag, ABC, AAP, TA}. Moreover
multiplicity results with polynomial type nonlinearity with
sign-changing weight functions using Nehari manifold and fibering
map analysis is also studied in many papers ( see refs.\cite{TA,
GA,DP,WU,WU6,WU9,WU10,WUFI, COA}). In this work we use fibering map analysis and Nehari manifold approach to solve the problem
\eqref{eq01}. The approach is not new but the results that we obtained are new. Our work is motivated by the work of Servadei and Valdinoci \cite{mp}, Brown and Zhang \cite{bz} and Afrouzi et al. \cite{GA}.

\noi First we define the space
{\small\[X_0= \left\{u|\;u:\mb R^n \ra\mb R \;\text{is measurable},
u|_{\Om} \in L^p(\Om), \left(u(x)- u(y)\right)\sqrt[p]{K(x-y)}\in
L^p(Q), u=0\;\mbox{on}\;\mb R^n\setminus \Om \right\},\]}
\noi where $Q=\mb R^{2n}\setminus(\mc C\Om\times \mc C\Om)$. In the next section, we study the properties of the $X_0$ in details.
\begin{Definition}\label{def1}
 A function $u \in X_0$ is a weak solution of \eqref{eq01}, if $u$
satisfies
\begin{align}\label{eq03}
\int_{Q}|u(x)-u(y)|^{p-2}&(u(x)-u(y))(v(x)-v(y)) K(x-y)dxdy\notag\\
&\quad\quad = \la \int_{\Om} |u|^{p-2} u v dx +  \int_{\Om}
b(x)|u|^{\ba-2}u v dx\;\; \fa\;\; v\in X_0.
\end{align}
\end{Definition}
\noi We define the Euler function $J_{\la}: X_0 \ra \mb R$ associated to the
problem \eqref{eq01} as
\[J_{\la}(u)= \frac{1}{p}\int_{Q}|u(x)-u(y)|^p K(x-y) dxdy -\frac{\la}{p}\int_{\Om}|u|^{p} dx -\frac{1}{\ba}\int_{\Om} b(x) |u|^{\ba}.\]
 Then $J_{\la}$ is Fr$\acute{e}$chet differentiable in $ X_0$ and
\begin{align*}
\ld J_{\la}^{\prime}(u),v
\rd=\int_{Q}&|u(x)-u(y)|^{p-2}(u(x)-u(y))(v(x)-v(y)) K(x-y)dxdy\\
&\quad- \la \int_{\Om} |u|^{p-2} u v dx -  \int_{\Om}b(x)
|u|^{\ba-2}u v dx,
\end{align*}
which shows that the weak solutions of \eqref{eq01} are exactly the
critical points of the functional $J_{\la}$.

In order to state our main result, we introduce some notations. The Nehari Manifold $\mc N_{\la}$ is defined by
\begin{equation*}
\mc N_{\la}= \left\{u\in X_0: \int_{Q}|u(x)-u(y)|^{p} K(x-y)dxdy- \la \int_{\Om} |u|^{p} dx -  \int_{\Om}b(x)
|u|^{\ba} dx =0
\right\}
\end{equation*}
and $\mc N_{\la}^{-}$, $\mc N_{\la}^{+}$ and $\mc N_{\la}^{0}$ are subset of $\mc N_{\la}$ corresponding to local minima,
local maxima and points of inflection of the fiber maps $t\mapsto J_{\la}(tu)$. For more details refer Section 2. Now we state the main result.
\noi In $p-$sublinear case$(1<\ba<p)$, we first studies the existence result for problem \eqref{eq01} with $\la<\la_1$ and the asymptotic behavior of these solutions as $\la\ra\la^{-}_{1}$. We have the following Theorem:
\begin{Theorem}\label{sb1}
For every $\la<\la_1$, problem \eqref{eq01} possesses at least one non-negative solution which is a minimizer for $J_{\la}$ on $\mc N_{\la}^{+}$. Moreover, if $\int_{\Om} b(x)\phi_{1}^{\ba} dx >0$, then
\begin{enumerate}
\item[(i)] $\ds{\lim_{\la\ra \la_{1}^{-}}\inf_{u\in\mc N_{\la}^{+}}J_{\la}(u)= -\infty}$.
\item[(ii)] If $\la_k\ra \la_{1}^{-}$ and $u_k$ is a minimizer of $J_{\la_k}$ on $\mc N_{\la}^{+}$, then $\ds{\lim_{k\ra\infty}} \|u_{k}\|= +\infty$.
\end{enumerate}
\end{Theorem}
Now, we state the multiplicity results for $\la>\la_1$ and the asymptotic behavior for these solutions as $\la\ra\la_{1}^{+}$.

\begin{Theorem}\label{sb2}
Suppose $\int_{\Om} b(x)\phi_{1}^{\ba} dx<0$, then there exists $\de_1>0$ such that the problem \eqref{eq01} has at least two non-negative solutions whenever $\la_1<\la<\la_1+ \de_1$, the two solutions are minimizers of $J_{\la}(u)$ on $\mc N_{\la}^{+}$ and $\mc N_{\la}^{-}$ respectively. Moreover, we have:
\begin{enumerate}
\item[(i)] $\ds{\lim_{\la\ra \la_{1}^{+}}\inf_{u\in\mc N_{\la}^{-}}}J_{\la}(u)= +\infty$.
\item[(ii)] If $\la_k\ra \la_{1}^{+}$ and $u_k$ is a minimizer of $J_{\la_k}$ on $\mc N_{\la}^{-}$, then $\ds{\lim_{k\ra\infty} \|u_{k}\|= +\infty}$.
\end{enumerate}
\end{Theorem}

\noi Next, we study the $p-$superlinear case$(p<\ba<p^*)$, in which we first study the existence result for problem \eqref{eq01} with $\la<\la_1$ and the asymptotic behavior of these solutions as $\la\ra\la^{-}_{1}$. We have the following Theorem:
\begin{Theorem}\label{sp1}
For every $\la<\la_1$, problem \eqref{eq01} possesses at least one non-negative solution which is a minimizer for $J_{\la}$ on $\mc N_{\la}^{-}$. Moreover, if $\int_{\Om} b(x)\phi_{1}^{\ba} dx >0$, then
\begin{enumerate}
\item[(i)] $\ds{\lim_{\la\ra \la_{1}^{-}}\inf_{u\in\mc N_{\la}^{-}}J_{\la}(u)=0}$.
\item[(ii)] If $\la_k\ra \la_{1}^{-}$ and $u_k$ is a minimizer of $J_{\la_k}$ on $\mc N_{\la}^{-}$, then $\ds{\lim_{k\ra\infty} u_{k}=0}$.
\end{enumerate}
\end{Theorem}
Next, we state the multiplicity result for $\la>\la_1$ and the asymptotic behavior for these solutions as $\la\ra\la_{1}^{+}$.

\begin{Theorem}\label{sp2}
Suppose $\int_{\Om} b(x)\phi_{1}^{\ba} dx<0$, then there exists $\de_1>0$ such that the problem \eqref{eq01} has at least two non-negative solutions whenever $\la_1<\la<\la_1+ \de_1$, the two solutions are minimizers of $J_{\la}(u)$ on $\mc N_{\la}^{+}$ and $\mc N_{\la}^{-}$ respectively. Moreover, let $u_k$ be minimizer of $J_{\la_{k}}$ on $\mc N_{\la}^{+}$ with $\la_{k}\ra \la_{1}^{+}$, then
\begin{enumerate}
\item[(i)] $u_{k}\ra 0$ as $k\ra\infty$.
\item[(ii)] $\frac{u_k}{\|u_k\|}\ra \phi_1$ in $X_0$ as $k\ra\infty$.
\end{enumerate}
\end{Theorem}
We should remark that the assumption $\int_{\Om} b(x)\phi_{1}^{\ba} dx<0$ is necessary for obtaining the existence result for problem \eqref{eq01}. In fact, the following theorem shows that we can't get a non-trivial solution by looking for minimizer of $J_{\la}$ on $\mc N_{\la}^{-}$ when $\int_{\Om} b(x)\phi_{1}^{\ba} dx>0$.
\begin{Theorem}\label{sp3}
Suppose $\int_{\Om} b(x)\phi_{1}^{\ba} dx>0$, then $\ds{\inf_{u\in \mc N_{\la}^{-}} J_{\la}(u)=0}$ for all $\la>\la_1$.
\end{Theorem}

\noi The paper is organized as follows: In section 2, we give some preliminaries results. In section 3, we study the behavior of Nehari manifold using
fibering map analysis for \eqref{eq01}. Section 4 contains the
existence of non-trivial solutions in $\mc N_{\la}^{+}$ and $\mc
N_{\la}^{-}$ and non-existence results in $p-$sublinear case. Section 5 contains the
existence and non-existence of solutions in $p-$superlinear case.

\noi We shall throughout use the following notations: The norm on
$X_0$ and $L^{p}(\Om)$ are denoted by $\|\cdot\|$ and $\|u\|_{p}$
respectively. The weak convergence is denoted by $\rightharpoonup$
and $\ra$ denotes strong convergence. We also define $u^{+}=\max(u,0)$ and $u^{-}=\max{(-u,0)}$.

\section{Functional Analytic Settings}

\noi In this section, we first define the function space and prove some properties which are useful to find the solution of the the problem \eqref{eq01}. For this we define $W^{\al,p}(\Om)$, the usual fractional Sobolev
space $W^{\al,p}(\Om):= \left\{u\in L^{p}(\Om); \frac{(u(x)-u(y))}{|x-y|^{\frac{n}{p}+\al}}\in L^{p}(\Om\times\Om)\right\}$ endowed with the norm
\begin{align}\label{2}
\|u\|_{W^{\al,p}(\Om)}=\|u\|_{L^p}+ \left(\int_{\Om\times\Om}
\frac{|u(x)-u(y)|^{p}}{|x-y|^{n+p\al}}dxdy \right)^{\frac 1p}.
\end{align}
To study fractional Sobolev space in details we refer \cite{hic}.

\noi Due to the non-localness of the operator $\mc L_K$ we define linear space as follows:
\[X= \left\{u|\;u:\mb R^n \ra\mb R \;\text{is measurable},
u|_{\Om} \in L^p(\Om)\;
 and\;  \left(u(x)- u(y)\right)\sqrt[p]{K(x-y)}\in
L^p(Q)\right\}\]

\noi where $Q=\mb R^{2n}\setminus(\mc C\Om\times \mc C\Om)$ and
 $\mc C\Om := \mb R^n\setminus\Om$. In case of $p=2$, the space $X$ was firstly introduced by Servadei and Valdinoci \cite{mp}. The space X is a normed linear space endowed with the norm
\begin{align}\label{1}
 \|u\|_X = \|u\|_{L^p(\Om)} +\left( \int_{Q}|u(x)- u(y)|^{p}K(x-y)dx
dy\right)^{\frac1p}.
\end{align}
 Then we define
 \[ X_0 = \{u\in X : u = 0 \;\text{a.e. in}\; \mb R^n\setminus \Om\}\]
with the norm
\begin{align}\label{01}
 \|u\|=\left(\int_{Q}|u(x)-u(y)|^{p}K(x-y)dx dy\right)^{\frac1p}
\end{align}
is a reflexive Banach space.
We notice that, even in the model case in which $K(x)=|x|^{n+p\al}$, the
norms in \eqref{2} and \eqref{1} are not same because $\Om\times
\Om$ is strictly contained in $Q$. Now we prove some properties of the spaces $X$ and $X_0$. Proof of these are easy to extend as in \cite{mp} but for completeness, we give the detail of proof.

\begin{Lemma}\label{l2}
Let $K :\mb R^n\setminus \{0\}\ra (0,\infty)$ be a function
satisfying $(b)$. Then
\begin{enumerate}
\item[1.] If $u\in X$ then $u\in W^{\al, p}(\Om)$ and moreover
\[\|u\|_{W^{\al,p}(\Om)}\leq c(\theta) \|u\|_{X}.\]
\item[2.] If $u\in X_0$ then $u\in W^{\al, p}(\mb R^n)$ and moreover
\[\|u\|_{W^{\al,p}(\Om)}\leq \|u\|_{W^{\al,p}(\mb R^n)} \leq c(\theta) \|u\|_{X}.\]
\end{enumerate}
 In both the cases $c(\theta)=\max\{1,
\theta^{-1/p}\}$, where $\theta$ is given in $(b)$.
\end{Lemma}

\begin{proof}
\begin{enumerate}
\item[1.] Let $u\in X$, then by $(b)$ we have
\begin{align*}
\int_{\Om\times\Om}\frac{|u(x)-u(y)|^{p}}{|x-y|^{n+p\al}} dx dy
&\leq \frac{1}{\theta}\int_{\Om\times\Om} |u(x)-u(y)|^{p} K(x-y) dx
dy\\
& \leq \frac{1}{\theta}\int_{Q} |u(x)-u(y)|^{p} K(x-y) dx dy<\infty.
\end{align*}
Thus
\begin{align*}
\|u\|_{W^{\al,p}}=\|u\|_{p}+
\left(\int_{\Om\times\Om}\frac{|u(x)-u(y)|^{p}}{|x-y|^{n+p\al}} dx
dy\right)^{\frac{1}{p}}\leq c(\theta) \|u\|_X.
\end{align*}
\item[2.] Let $u\in X_0$ then $u=0$ on $\mb R^n\setminus \Om$. So
$\|u\|_{L^{2}(\mb R^n)}= \|u\|_{L^{2}(\Om)}$. Hence
\begin{align*}
\int_{\mb R^{2n}}\frac{|u(x)-u(y)|^{p}}{|x-y|^{n+p\al}} dx dy &=
\int_{Q} \frac{|u(x)-u(y)|^{p}}{|x-y|^{n+p\al}} dx dy\\
& \leq \frac{1}{\theta}\int_{Q} |u(x)-u(y)|^{p} K(x-y) dx dy<
+\infty,
\end{align*}
\end{enumerate}
as required.\QED
\end{proof}

\begin{Lemma}\label{l3}
Let $K :\mb R^n\setminus \{0\}\ra (0,\infty)$ be a function
satisfying $(b)$. Then there exists a positive constant $c$
depending on $n$ and $\al$ such that for every $u\in X_0$, we have
\[\|u\|_{L^{p^*}(\Om)}^{p}= \|u\|_{L^{p^*}(\mb R^n)}^{p}\leq c\int_{\mb R^{2n}}\frac{|u(x)-u(y)|^{p}}{|x-y|^{n+p\al}} dx dy, \]
where $p^*=\frac{np}{n-p\al}$ is fractional critical Sobolev exponent.
\end{Lemma}
\begin{proof}
Let $u\in X_0$ then by Lemma \ref{l2}, $u\in W^{\al,p}(\mb R^n)$. Also we know that $W^{\al,p}(\mb R^n)\hookrightarrow L^{p^*}(\mb
R^n)$ (see \cite{hic}). Then we have,
\[\|u\|_{L^{p^*}(\Om)}^{p}= \|u\|_{L^{p^*}(\mb R^n)}^{p}\leq c\int_{\mb R^{2n}}\frac{|u(x)-u(y)|^{p}}{|x-y|^{n+p\al}} dx dy  \]
and hence the result. \QED
\end{proof}
\begin{Lemma}\label{l4}
Let $K :\mb R^n\setminus \{0\}\ra (0,\infty)$ be a function
satisfying $(b)$. Then there exists $C>1$, depending only on $n$,
$\al$, $p$, $\theta$ and $\Om$ such that for any $u\in X_0$,
\[\int_{Q} |u(x)-u(y)|^{p} K(x-y) dx dy\leq \|u\|_{X}^{p} \leq C\int_{Q} |u(x)-u(y)|^{p} K(x-y) dx dy.\]
i.e. \begin{align}\label{ee1} \|u\|^{p} = \int_{Q}
|u(x)-u(y)|^{p} K(x-y) dx dy
\end{align}
 is a norm on $X_0$ and equivalent to the norm on $X$.
\end{Lemma}
\begin{proof}
Clearly $\|u\|_{X}^{p}\geq \int_{Q} |u(x)-u(y)|^{p} K(x-y) dx dy$. Now by Lemma \ref{l3} and $(b)$, we get
\begin{align*}
\|u\|_{X}^{p}&= \left(\|u\|_{p} + \left(\int_{Q}
|u(x)-u(y)|^{p} K(x-y) dx dy \right)^{1/p}\right)^p \\
&\leq 2^{p-1} \|u\|_{p}^{p} + 2^{p-1} \int_{Q} |u(x)-u(y)|^{p} K(x-y) dx dy\\
&\leq  2^{p-1} |\Om|^{1-\frac{p}{p^*}} \|u\|_{p^*}^{p} + 2^{p-1} \int_{Q} |u(x)-u(y)|^{p} K(x-y) dx dy\\
&\leq  2^{p-1} \;c|\Om|^{1-\frac{p}{p^*}} \int_{\mb R^{2n}}\frac{|u(x)-u(y)|^{p}}{|x-y|^{n+p\al}}dx dy + 2^{p-1} \int_{Q} |u(x)-u(y)|^{p} K(x-y) dx dy\\
&\leq 2^{p-1}
\left(\frac{c|\Om|^{1-\frac{p}{p^*}}}{\theta}+1\right)\int_{Q}
|u(x)-u(y)|^{p} K(x-y) dx dy\\
&= C \int_{Q} |u(x)-u(y)|^{p} K(x-y) dx dy,
\end{align*}
where $C>1$ as required. Now we show that \eqref{ee1} is a norm on
$X_0$. For this we need only to show that if $\|u\|=0$ then
$u=0$ a.e. in $\mb R^n$ as other properties of norm are obvious. Indeed, if $\|u\|=0$ then
 $ \int_{Q}|u(x)- u(y)|^{p}K(x-y)dx dy=0$ which implies that $u(x)=
u(y)$ a.e in $Q$. Therefore, $u$ is constant in $Q$ and hence $u=c\in \mb R$
a.e in $\mb R^n$. Also by definition of $X_0$, we have $u=0$ on $\mb
R^n\setminus \Om$. Thus $u=0$ a.e. in $\mb R^n$.\QED

\end{proof}
\begin{Lemma}
Let $K :\mb R^n\setminus \{0\}\ra (0,\infty)$ be a function
satisfying $(b)$ and let $\{u_k\}$ be a bounded sequence in $X_0$.
Then there exists $u\in L^{\ba}(\mb R^n)$ such that up to a
subsequence, $u_k\ra u$ strongly in $L^{\ba}(\mb R^n)$ as $k\ra \infty$ for any
$\ba\in [1, p^*)$.
\end{Lemma}

\begin{proof}
Let $\{u_k\}$ is bounded in $X_0$. Then by Lemmas \ref{l2} and \ref{l4},
$\{u_k\}$ is bounded in $W^{\al,p}(\Om)$ and in $L^{p}(\Om)$. Also
by assumption on $\Om$ and [4, Corollary 7.2], there exists $u\in
L^{\ba}(\Om)$ such that up to a subsequence $u_k\ra u$ strongly in $
L^{\ba}(\Om)$ as $k\ra\infty$ for any $\ba\in[1,p^*)$. Since $u_k=0$ on
$\mb R^n\setminus \Om$, we can define $u:=0$ in $\mb R^n\setminus
\Om$. Then we get $u_k\ra u$ in $ L^{\ba}(\mb R^n)$.\QED
\end{proof}

\section{Nehari Manifold and fibering map analysis}
In this section, we introduce the Nehari Manifold and exploit the relationship between Nehari Manifold and fibering map.
\noi Now the Euler functional $J_{\la}: X_0\ra \mb R$ is defined as
\[J_{\la}(u)= \frac{1}{p}\int_{Q}|u(x)-u(y)|^p K(x-y) dxdy -\frac{\la}{p}\int_{\Om}|u|^{p} dx -\frac{1}{\ba}\int_{\Om} b(x) |u|^{\ba}.\]
If $J_{\la}$ is bounded below on $X_0$ then minimizers of $J_{\la}$ on $X_0$ become the critical point of $J_{\la}$. Here $J_{\la}$ is not
bounded below on $X_0$ but is bounded below on appropriate subset of $X_0$ and minimizer on this set(if it exists) give rise to solutions of the problem \eqref{eq01}.
Therefore in order to obtain the existence results, we
introduce the Nehari manifold
\begin{equation*}
\mc N_{\la}= \left\{u\in X_0: \ld J_{\la}^{\prime}(u),u\rd=0
\right\}=\left\{u\in X_0 : \phi_{u}^{\prime}(1)=0\right\}
\end{equation*}
where $\ld\;,\; \rd$ denotes the duality between $X_0$ and its dual
space. Thus $u\in \mathcal N_{\la}$ if and only if
\begin{equation}\label{eq2}
\int_{Q} |u(x)-u(y)|^p K(x-y) dxdy - \la \int_{\Om} |u|^{p} dx-
\int_{\Om} b(x)|u|^{\ba}dx =0 .
\end{equation}
We note that $\mathcal N_{\la}$ contains every solution of
\eqref{eq01}. Now as we know that the Nehari manifold is closely
related to the behavior of the functions $\phi_u: \mb R^+\ra \mb R$
defined as $\phi_{u}(t)=J_{\la}(tu)$. Such maps are called fiber
maps and were introduced by Drabek and Pohozaev in \cite{DP}. For
$u\in X_0$, we have
\begin{align*}
\phi_{u}(t) &= \frac{t^p}{p} \|u\|^p- \frac{\la
 t^{p}}{p}\int_{\Om}|u|^{p} dx - \frac{t^{\ba}}{\ba}\int_{\Om} b(x) |u|^{\ba} dx ,\\
\phi_{u}^{\prime}(t) &=  t^{p-1}\|u\|^{p}- {\la
  t^{p-1}}\int_{\Om} |u|^{p} dx  - t^{\ba-1} \int_{\Om} b(x) |u|^{\ba} dx,\\
\phi_{u}^{\prime\prime}(t) &= (p-1) t^{p-2}\|u\|^{p}-  \la(p-1)
 t^{p-2} \int_{\Om} |u|^{p} dx - (\ba-1) t^{\ba-2}\int_{\Om} b(x) |u|^{\ba} dx.
\end{align*}
Then it is easy to see that $tu\in \mathcal N_{\la}$ if and only if
$\phi_{u}^{\prime}(t)=0$ and in particular, $u\in \mc N_{\la}$ if
and only if $\phi_{u}^{\prime}(1)=0$. Thus it is natural to split
$\mathcal N_{\la}$ into three parts corresponding to local minima,
local maxima and points of inflection. For this we set
\begin{align*}
\mathcal N_{\la}^{\pm}&:= \left\{u\in \mc N_{\la}:
\phi_{u}^{\prime\prime}(1)
\gtrless0\right\} =\left\{tu\in X_0 : \phi_{u}^{\prime}(t)=0,\; \phi_{u}^{''}(t)\gtrless  0\right\},\\
\mathcal N_{\la}^{0}&:= \left\{u\in \mc N_{\la}:
\phi_{u}^{\prime\prime}(1) = 0\right\}=\left\{tu\in X_{0} :
\phi_{u}^{\prime}(t)=0,\; \phi_{u}^{''}(t)= 0\right\}.
\end{align*}
We also observe that if $tu\in\mc N_{\la}$ then
$\phi_{u}^{\prime\prime}(t)=(p-\ba) t^{\ba-2}\int_{\Om}
b(x)|u|^{\ba}dx$.
\noi Now we describe the behavior of the fibering map $\phi_{u}$
according to the sign of $E_{\la}(u):= \|u\|^{p}- \la\int_{\Om} |u|^{p} dx$ and $B(u):=
\int_{\Om} b(x)|u|^{\ba} dx$. Define
\[E^{\pm}_{\la}:= \{u\in X_0: \|u\|=1, E_{\la}(u)\gtrless 0\}, \quad \; B^{\pm}:= \{u\in X_0: \|u\|=1, B(u)\gtrless 0\}, \]
\[E^{0}_{\la}:= \{u\in X_0: \|u\|=1, E_{\la}(u)= 0\},\;\quad\; B^{0}:= \{u\in X_0: \|u\|=1, B(u)= 0\}.\]

\noi Case 1: $u\in E_{\la}^{-}\cap B^{+}$.\\
In this case $\phi_{u}(0)=0$, $\phi_{u}^{\prime}(t)<0$ $\fa$ $t>0$
which means that $\phi_{u}$ is strictly decreasing and so it has no
critical point.

\noi Case 2: $u\in E_{\la}^{+}\cap B^{-}$.\\
In this case $\phi_{u}(0)=0$, $\phi_{u}^{\prime}(t)>0$ $\fa$ $t>0$
which implies that $\phi_{u}$ is strictly increasing and hence no
critical point.\\
Now the other cases depend on $\ba$ as the behavior of $\phi_{u}$ changes according to $1<\ba<p$ or $p<\ba<p^{*}$.\\
\noi Case 3: $u\in E_{\la}^{+}\cap B^{+}$.\\
In {\bf $p-$sublinear} case$(1<\ba<p)$, $\phi_{u}(0)=0$, $\phi_{u}(t)\ra +\infty$ as $t\ra \infty$ and $\phi_{u}(t)<0$ for small $t>0$ as $u\in E_{\la}^{+}\cap B^{+}$. Also $\phi_{u}^{\prime}(t)=0$ when $t(u)=
\left[\frac{\int_{\Om} b(x)|u|^{\ba} dx}{\|u\|^{p} - \la \int_{\Om}
|u|^{p} dx}\right]^{\frac{1}{p-\ba}}$. Thus $\phi_{u}$
has exactly one critical point $t(u)$, which is a global minimum
point. Hence $t(u)u \in \mc N^{+}_{\la}$.\\

\noi In $p-$superlinear case$(p <\ba < p^{*})$, $\phi_{u}(0)=0$, $\phi_{u}(t)>0$ for small $t>0$ as $u\in E_{\la}^{+}\cap B^{+}$, $\phi_{u}(t)\ra -\infty$ as $t\ra
\infty$ and $\phi_{u}^{\prime}(t)=0$ when \[t(u)= \left[\frac{\|u\|^{p}
- \la \int_{\Om} |u|^{p} dx}{\int_{\Om} b(x)|u|^{\ba}
dx}\right]^{\frac{1}{\ba-p}}.\]
This implies that $\phi_{u}$ has exactly one critical
point $t(u)$, which is a global maximum point. Hence $t(u)u \in \mc
N^{-}_{\la}$.\\

\noi Case 4: $u\in E_{\la}^{-}\cap B^{-}$.\\
In {\bf $p-$sublinear} case, $\phi_{u}(0)=0$, $\phi_{u}(t)>0$ for small $t>0$ as $u\in E_{\la}^{-}\cap B^{-}$, $\phi_{u}(t)\ra -\infty$ as $t\ra
\infty$ and $\phi_{u}^{\prime}(t)=0$ when \[t(u)=
\left[\frac{\int_{\Om} b(x)|u|^{\ba} dx}{\|u\|^{p} - \la \int_{\Om}
|u|^{p} dx}\right]^{\frac{1}{p-\ba}}.\] This implies that
$\phi_{u}$ has exactly one critical point $t(u)$, which is a global
maximum point. Hence $t(u)u \in \mc N^{-}_{\la}$.\\
In {\bf $p-$superlinear} case, $\phi_{u}(0)=0$, $\phi_{u}(t)<0$ for small $t>0$ as $u\in E_{\la}^{-}\cap B^{-}$, $\phi_{u}(t)\ra +\infty$ as $t\ra \infty$ and $\phi_{u}^{\prime}(t)=0$, when $t(u)=
\left[\frac{\|u\|^{p} - \la \int_{\Om} |u|^{p} dx}{\int_{\Om}
b(x)|u|^{\ba} dx}\right]^{\frac{1}{\ba-p}}$. Thus $\phi_{u}$
has exactly one critical point $t(u)$, which is a global minimum
point. Hence $t(u)u \in \mc N^{+}_{\la}$.

\noi The following Lemma shows that the minimizers for $J_{\la}$ on $\mc N_{\la}$ are often critical points of $J_{\la}$.
\begin{Lemma}\label{le10}
Let $u$ be a local minimizer for $J_{\la}$ on any of above subsets of $\mc
N_{\la}$ such that $u\notin \mc N_{\la}^{0}$, then $u$ is a critical
point for $J_{\la}$.
\end{Lemma}
\proof Since  $u$ is a minimizer for $J_{\la}$
under the constraint $I_{\la}(u):=\ld J_{\la}^{\prime}(u),u\rd = 0$, by the theory of Lagrange multipliers, there exists $\mu \in
\mb R$ such that $ J _{\la}^{\prime}(u)= \mu I_{\la}^{\prime}(u)$.
Thus $\ld J_{\la}^{\prime}(u),u\rd= \mu\;\ld
I_{\la}^{\prime}(u),u\rd = \mu \phi_{u}^{\prime\prime}(1)= 0$, but
$u\notin \mc N_{\la}^{0}$ and so $\phi_{u}^{\prime\prime}(1) \ne 0$.
Hence $\mu=0$ completes the proof.\QED


\noi Let $\la_1$ be the smallest eigenvalue of $-\mc L_K$ which is characterized as
\[\la_1 = \inf_{u\in
X_{0}}\left\{\int_{Q}|u(x)-u(y)|^p K(x-y)dxdy : \int_{\Om}
|u|^p=1\right\}.\]
Let $\phi_1$ denotes the eigenfunction corresponding to the the eigenvalue $\la_1$. That is $(\la_1, \phi_1)$ satisfies
\begin{equation*}
 \quad \left.
\begin{array}{lr}
 \quad -\mc L_{K}u(x) = \la |u(x)|^{p-2}u(x) \; \text{in}\;
\Om \\
 \quad \quad u = 0 \; \mbox{in}\; \mb R^n \setminus\Om.\\
\end{array}
\quad \right\}
\end{equation*}
 Then
\begin{align}\label{a1}
\int_{Q}|u(x)-u(y)|^{p}K(x-y)dx
dy -\la \int_{\Om}|u|^{p}dx \geq (\la_1-\la) \int_{\Om}|u|^{p}dx\;
\mbox{for all}\; u\in X_0.
\end{align}
\noi Moreover, in \cite{gf}, it is proved that $\la_1$ is simple.
\noi We distinguish the $p-$sublinear and $p-$superlinear case respectively. In the following section we first study the $p-$sublinear case.
\section{$p-$Sublinear Case($1<\ba<p$)}
In this section, we give the detail proof of Theorem \ref{sb1} and \ref{sb2}.
Using \eqref{a1} we have
\begin{align*}
J_{\la}(u)&\geq \frac{1}{p}(\la_1 -\la)\int_{\Om}|u|^p dx
-\frac{1}{\ba}\int_{\Om} b(x)|u|^{\ba} dx\\
&\geq \frac{1}{p}(\la_1 -\la)\int_{\Om}|u|^p dx
-\frac{\overline{b}}{\ba}|\Om|^{1-\frac{\ba}{p}}\left(\int_{\Om}|u|^{p}
dx\right)^{\frac{\ba}{p}}
\end{align*}
where $\overline{b}=\ds{\sup_{x\in \Om} b(x)}$. Hence $J_{\la}$ is bounded below on
$X_0$, when $\la<\la_1$. When $\la>\la_1$, it is easy to see that $J_{\la}(t\phi_1)\ra -\infty$ as $t\ra \infty$. Therefore $J_{\la}$ is not
bounded below on $X_0$. But we show that it is bounded below on the some subset of $\mc N_{\la}$. Also in this case i.e. ($1<\ba<p$), from the definition of $\mc N_{\la}^{\pm}$ and $\mc N_{\la}^{0}$, it is not difficult to see that
\[\mc N_{\la}^{\pm}= \left\{u\in \mc N_{\la}: \int_{\Om} b(x)|u|^{\ba}dx\gtrless0\right\},\;\; \mc N_{\la}^{0}= \left\{u\in \mc N_{\la}: \int_{\Om} b(x)|u|^{\ba}dx=0\right\}.\]
Now on $\mc N_{\la}$, $J_{\la}(u)= \left(\frac{1}{p}-\frac{1}{\ba}\right) \int_{\Om}b(x)|u|^{\ba} dx =\left(\frac{1}{p}-\frac{1}{\ba}\right)(\|u\|^{p}- \la \int_{\Om}|u|^p dx)$.
Then we note that $J_{\la}(u)$ changes sign in $\mc N_{\la}$ but this is true only if both $\mc N_{\la}^{+}$ and $\mc N_{\la}^{-}$ are nonempty. We
have $J_{\la}(u)>0$ on $\mc N_{\la}^{-}$ and $J_{\la}(u)<0$ on $\mc
N_{\la}^{+}$.\\

 \noi When $0<\la<\la_1$, $\|u\|^p - \la \int_{\Om} |u|^{p} dx>0$ for all $u\in X_0$. This implies that
$E_{\la}^{+}=\{u\in X_0 : \|u\|=1\}$, $E_{\la}^{-}$ and
$E_{\la}^{0}$ are empty sets. Thus $\mc N_{\la}^{-}=\emptyset= \mc N_{\la}^{0}$ and $\mc N_{\la}=\mc N_{\la}^{+}\cup
\{0\}$. If $\la>\la_1$ then
$$\int_{Q} |\phi_{1}(x)-\phi_{1}(y)|^p K(x-y) dxdy - \la
\int_{\Om} |\phi_{1}|^{p} dx = (\la_1 -\la) \int_{\Om}
|\phi_{1}|^{p} dx <0$$ and so $\phi_1\in E^{-}_{\la}$. Hence, for $\la=\la_1$, we have $E_{\la}^{-}=\emptyset$ and $E_{\la}^{0}=\{\phi_1\}$. And moreover when $\la>\la_1$, $E_{\la}^{-}$ is non-empty and gets bigger as
$\la$ increases. Now we discuss the vital role played by the condition $E^{-}_{\la}\subset B^{-}$ to determine the nature of Nehari manifold. In view of above discussion, this condition is always satisfied when $\la<\la_1$ and may or may not be satisfied when $\la>\la_1$.

\begin{Theorem}\label{tt1}
Suppose there exists $\la_0$ such that for all $\la<\la_0$,
$E^{-}_{\la}\subset B^{-}$. Then for all $\la<\la_0$ we have the
following
\begin{enumerate}
\item[(1)] $E^{0}_{\la}\subseteq B^{-}$ and so $E^{0}_{\la}\cap B^{0}=
\emptyset$.
\item[(2)]  $\mc N_{\la}^{+}$ is bounded.
\item[(3)] $0\not\in \overline{\mc N_{\la}^{-}}$ and $\mc N_{\la}^{-}$ is closed.
\item[(4)] $\overline{\mc N_{\la}^{+}}\cap\mc N_{\la}^{-}=\emptyset$.
\end{enumerate}
\end{Theorem}

\begin{proof} $(1)$ Suppose this is not true. Then there exists $u\in
E^{0}_{\la}$ such that $u\not\in B^{-}$. If we take $\mu$ such
that $\la<\mu<\la_0$, then $u\in E^{-}_{\mu}$ and so
$E^{-}_{\mu}\not\subseteq B^{-}$ which gives a contradiction. Thus $E^{0}_{\la}\subseteq B^{-}$ and so $E^{0}_{\la}\cap B^{0}=
\emptyset$.
\begin{enumerate}
\item[$(2)$] Suppose $\mc N_{\la}^{+}$ is not bounded. Then there exists a
sequence $\{u_k\}\subseteq \mc N_{\la}^{+}$ such that
$\|u_k\|\ra\infty$ as $k\ra \infty$. Let $v_k=\frac{u_k}{\|u_k\|}$.
Then we may assume that up to a subsequence $v_k\rightharpoonup v_0$ weakly in $X_0$ and
so $v_k\ra v_0$ strongly in $L^{p}(\Om)$ for every $1\leq p<p^*$.
Also $\int_{\Om}b|v_k|^{\ba}>0$ as $u_k\in \mc N_{\la}^{+}$ and
so $\int_{\Om}b|v_0|^{\ba}\geq 0$. Since $u_k\in \mc N_{\la}$, we
have
\begin{align*}
\int_{Q} |u_k(x)-u_k(y)|^p K(x-y) dxdy - \la \int_{\Om} |u_k|^{p} dx
= \int_{\Om} b(x)|u_k|^{\ba}dx,
\end{align*}
which implies
\[\|v_k\|^p  - \la \int_{\Om} |v_k|^{p} dx
= \frac{1}{\|u_k\|^{p-\ba}}\int_{\Om} b(x)|v_k|^{\ba}dx \lra
0\;\mbox{as}\; k\ra\infty.\]
\noi Suppose $v_k\not\ra v_0$ strongly
in $X_0$. Then $\|v_0\|^p <\ds\liminf_{k\ra\infty} \|v_k\|^p$
and so
\[\|v_0\|^p -\la\int_{\Om}|v_0|^{p}dx<\lim_{k\ra\infty} \int_{Q} |v_k(x)-v_k(y)|^p K(x-y) dxdy-\la\int_{\Om}|v_k|^{p}dx =0,\]
which implies that $\|v_0\|\not= 0$. If not, then we get $0<0$, a contradiction. Thus $\frac{v_0}{\|v_0\|}\in E^{-}_{\la} \subset B^{-}$ which is
a contradiction as $\int_{\Om} b(x)|v_0|^{\ba}dx\geq 0$. Hence $v_k\ra
v_0$ strongly in $X_0$. Thus $\|v_0\|=1$ and
\[\|v_0\|^p -\la\int_{\Om}|v_0|^{p}dx= \lim_{k\ra\infty} \int_{Q} |v_k(x)-v_k(y)|^p K(x-y) dxdy-\la\int_{\Om}|v_k|^{p}dx =0.\]
So $v_0\in E^{0}_{\la}\subseteq B^{-}$ by $(1)$, which is again a contradiction as $\int_{\Om} b |v_0|^{\ba} dx \geq 0$.
Hence $\mc N^{+}_{\la}$ is bounded.

\item[$(3)$] Suppose $0\in\overline{\mc N_{\la}^{-}}$. Then there exists
a sequence $\{u_k\}\subseteq \mc N_{\la}^{-}$ such that $\ds\lim_{k\ra\infty} u_k
=0$ in $X_0$. Let $v_k=\frac{u_k}{\|u_k\|}$. Then up to a subsequence
$v_k\rightharpoonup v_0$ weakly in $X_0$ and $v_k\ra v_0$ strongly in
$L^{p}(\Om)$. As $u_k\in \mc N_{\la}^{-}$, we have
\[\int_{Q} |v_k(x)-v_k(y)|^p K(x-y) dxdy - \la \int_{\Om} |v_k|^{p} dx
= \frac{1}{\|u_k\|^{p-\ba}}\int_{\Om} b(x)|v_k|^{\ba}dx\leq 0.\]
Since the left hand side is bounded, it follows that
$\int_{\Om} b(x)|v_0|^{\ba} = \ds\lim_{k\ra\infty}\int_{\Om} b(x)|v_k|^{\ba} =0$.
Now suppose that $v_k\ra v_0$ strongly in $X_0$. Then $\|v_0\|=1$ and so
$v_0\in B_0$. Moreover $\|v_0\|^p - \la \int_{\Om} |v_0|^{p} dx = \ds\lim_{k\ra\infty} \|v_k\|^p - \la \int_{\Om} |v_k|^{p} dx\leq 0,$
which implies that $v_0\in E_{\la}^{0}$ or $E_{\la}^{-}$. Hence
$v_0\in B^{-}$ which is a contradiction. Hence we must have $v_k\not\ra
v_0$ in $X_0$. Thus $\|v_0\|^p - \la \int_{\Om} |v_0|^{p} dx <\ds\lim_{k\ra \infty}\|v_k\|^p - \la \int_{\Om} |v_k|^{p} dx\leq 0,$ which implies that $\|v_0\|\not= 0$. If $\|v_0\|=0$, then we get $0<0$, a contradiction.
Hence $\frac{v_0}{\|v_0\|}\in E^{-}_{\la}\cap B^0$, which is
impossible so $0\not\in \oline{\mc N_{\la}^{-}}$.

\noi We now show that $\mc N_{\la}^{-}$ is a closed set. Let
$\{u_k\}\subseteq \mc N_{\la}^{-}$ be such that $u_k \ra u$ strongly in $X_0$. Then
$u\in \overline{\mc N_{\la}^{-}}$ and so $u\not\equiv 0$. Moreover,
 $\|u\|^p - \la \int_{\Om} |u|^{p} dx= \int_{\Om} b(x)|u|^{\ba}dx \leq 0.$ If both the integral  equal to
zero, then $\frac{u}{\|u\|}\in E^{0}_{\la}\cap B^{0}$, which gives a
contradiction to $(1)$. Hence both the integral must be negative, so $u\in \mc N^{-}_{\la}$. Thus $\mc N_{\la}^{-}$ is closed.

\item[$(4)$] Let $u\in \overline{\mc N_{\la}^{+}}\cap\mc N_{\la}^{-}$. Then
$0\not\equiv u\in \mc N_{\la}^{-}$ and moreover
\begin{align*}
\int_{Q} |u(x)-u(y)|^p K(x-y) dxdy - \la \int_{\Om} |u|^{p} dx =
\int_{\Om} b(x)|u|^{\ba}dx =0 .
\end{align*}
Thus $\frac{u}{\|u\|}\in E^{0}_{\la}\cap B^{0}$, which is a
contradiction and hence the result.\QED
\end{enumerate}
\end{proof}

\begin{Lemma}\label{le1}
Suppose there exists $\la_0$ such that for all $\la<\la_0$,
$E^{-}_{\la}\subset B^{-}$. Then for all $\la<\la_0$ we have,
\begin{enumerate}
\item[$(i)$] $J_{\la}$ is bounded below on $\mc N_{\la}^{+}$.
\item[$(ii)$] $J_{\la}$ is bounded below on $\mc N_{\la}^{-}$ and moreover
$\ds\inf_{u\in \mc N_{\la}^{-}}J_{\la}(u)>0$ provided $\mc N_{\la}^{-}$
is non-empty.
\end{enumerate}
\end{Lemma}
\begin{proof}$(i)$ It follows from the fact that $\mc N_{\la}^{+}$ is bounded.\\
$(ii)$ Suppose $\ds\inf_{u\in \mc N_{\la}^{-}}J_{\la}(u)=0$. Then there exists
a sequence $\{u_k\}\subseteq \mc N_{\la}^{-}$ such that
$J_{\la}(u_k)\ra 0$ as $k\ra\infty$, i.e.
\[\|u_k\|^p - \la \int_{\Om} |u_k|^{p}
dx\ra 0\; \mbox{and}\;  \int_{\Om} b(x)|u_k|^{\ba}dx \ra 0\;\mbox{ as}\; k\ra\infty.\]
Let $v_k=\frac{u_k}{\|u_k\|}$. Then, since $0\not\in \overline{\mc
N_{\la}^{-}}$, $\{\|u_k\|\}$ is bounded away from zero, so
\[\lim_{k\ra\infty} \int_{\Om} b(x)|v_k|^{\ba} dx =0\;\mbox{ and}\; \lim_{k\ra\infty}\left(\|v_k\|^p - \la
\int_{\Om} |v_k|^{p} dx\right) = 0.\]
As $v_k$ is bounded in $X_0$, we may assume that up to a subsequence still denoted by $v_k$ such that $v_k\rightharpoonup v_0$ weakly in $X_0$
and $v_k\ra v_0$ strongly in $L^{p}(\Om)$. Then $\int_{\Om}
b(x)|v_0|^{\ba} dx =0$.\\
If $v_k\ra v_0$ strongly in $X_0$ then we
have $\|v_0\|=1$ and $\|v_0\|^p - \la \int_{\Om} |v_0|^{p} dx = 0$.
i.e. $v_0\in E^{0}_{\la}$. Whereas if, $v_k\not\ra v_0$ then $\|v_0\|^p - \la \int_{\Om} |v_0|^{p} dx < 0$
 i.e $\frac{v_0}{\|v_0\|}\in E^{-}_{\la}$.
In both the cases, we also have $\frac{v_0}{\|v_0\|}\in B^0$, which is
a contradiction. Hence $\ds\inf_{u\in \mc N_{\la}^{-}}J_{\la}(u)>0$.\QED
\end{proof}

\begin{Theorem}\label{t2}
Suppose there exists $\la_0$ such that $E^{-}_{\la}\subseteq B^{-}$ for all $\la<\la_0$. Then for all
$\la<\la_0$, we have the following \\
$(i)$ there exists a minimizer for $J_{\la}$ on $\mc N_{\la}^{+}$.\\
$(ii)$ there exists a minimizer for $J_{\la}$ on $\mc N_{\la}^{-}$
provided $E^{-}_{\la}$ is non empty.
\end{Theorem}

\begin{proof}
$(i)$ By Lemma \ref{le1}, $J_{\la}$ is bounded below on $\mc N^{+}_{\la}$.
Let $\{u_k\}\subseteq \mc N_{\la}^{+}$ be a minimizing sequence,
i.e. $\ds\lim_{k\ra\infty} J_{\la}(u_k)=\ds\inf_{u\in \mc
N_{\la}^{+}}J_{\la}(u)<0$ as $J_{\la}(u)<0$ on $\mc N_{\la}^{+}$. Since $\mc N_{\la}^{+}$ is bounded,
we may assume that up to a subsequence still denoted by $\{u_k\}$ such
that $u_k\rightharpoonup u_0$ weakly in $X_0$ and $u_k\ra u_0$ strongly in
$L^{p}(\Om)$. Since $J_{\la}(u_k)= \left(\frac{1}{p}
-\frac{1}{\ba}\right)\int_{\Om} b(x)|u_k|^{\ba} dx $. It follows
that $\int_{\Om} b(x)|u_0|^{\ba} dx= \ds\lim_{k\ra\infty}\int_{\Om}
b(x)|u_k|^{\ba} dx>0$ and so $u_0 \not\equiv 0$ a.e. in $\mb R^n$ and
$\frac{u_0}{\|u_0\|}\in B^{+}$. Also by Theorem \ref{tt1}, $\frac{u_0}{\|u_0\|}\in
E^{+}_{\la}$. Thus by the fibering map analysis, $\phi_{u_0}$ has a unique minimum at $t(u_0)$ such that
$t(u_0)u_0\in \mc N_{\la}^{+}$. Now we claim that $u_k\ra u_0$ strongly in $X_0$. Suppose $u_k\not\ra u_0$ in $X_0$. Then
\begin{align*}
\|u_0\|^p - \la \int_{\Om} |u_0|^{p}< \lim_{k\ra\infty}(\|u_k\|^p - \la
\int_{\Om} |u_k|^{p} dx)= \lim_{k\ra\infty} \int_{\Om} b(x)|u_k|^{\ba}dx= \int_{\Om}
b(x)|u_0|^{\ba} dx
\end{align*}
and so $t(u_0)>1$. Hence
\[J_{\la}(t(u_0) u_{0}) < J_{\la}(u_{0})< \lim_{k\ra \infty} J_{\la}(u_k)= \inf_{u\in\mc N_{\la}^{+}}J_{\la}(u),\]
which is a contradiction. Thus we must have $u_k\ra u_{0}$ in
$X_0$, $u_{0}\in \mc N_{\la}$ and $u_{0}\in \mc
N_{\la}^{+}$. If $u_0\in \mc N_{\la}^{0}$ then $\int_{\Om}
b(x)|u_0|^{\ba} dx =0$ and $\|u_0\|^p -
\la \int_{\Om} |u_0|^{p} dx=0$. This implies that $0\not\equiv u_0\in E^{0}_{\la}\cap B^{0}$, a contradiction as $E^{-}_{\la}\cap B^0=\emptyset$, which is proved in Theorem \ref{tt1} $(1)$.\\


\noi $(ii)$ Let $\{u_k\}$ be a minimizing sequence for $J_{\la}$ on $\mc
N_{\la}^{-}$. Then by Lemma \ref{le1}, we must have
$\ds\lim_{k\ra\infty} J_{\la}(u_k)=\inf_{u\in \mc
N_{\la}^{-}}J_{\la}(u)>0.$ Now we claim that $\{u_k\}$ is a bounded
sequence. Suppose this is not true. Then there exists a subsequence $\{u_k\}$ such that $\|u_k\|\ra \infty$ as
$k\ra\infty$. Let $v_k=\frac{u_k}{\|u_k\|}$. Since $\{J_{\la}(u_k)\}$ is
bounded, it follows that $\{\int_{\Om} b(x)|u_k|^{\ba} dx\}$ and
$\{\|u_k\|^p -\la \int_{\Om}|u_k|^{p} dx\}$
are bounded and so
\[\lim_{k\ra\infty}\int_{Q} |v_k(x)- v_k(y)|^p K(x-y) dxdy -\la \int_{\Om} |v_k|^{p} dx =\lim_{k\ra\infty} \int_{\Om}b(x)|v_k|^{\ba} dx=0.\]
Since $\{v_k\}$ is bounded, we may assume that
$v_k\rightharpoonup v_0$ weakly in $X_0$ and $v_k\ra v_0$ strongly in
$L^{p}(\Om)$ so that $\int_{\Om} b(x)|v_0|^{\ba}=0$. If $v_k\ra v_0$
strongly in $X_0$ then it is easy to see that $v_0\in
E^{0}_{\la}\cap B^0$ which gives a contradiction by Theorem
\ref{tt1} $(1)$. Hence $v_k\not\ra v_0$ in $X_0$ and so
$$\|v_0\|^p -\la \int_{\Om}|v_0|^{p} dx<\lim_{k\ra\infty}\int_{Q} |v_k(x)-v_k(y)|^p K(x-y) dxdy-\la \int_{\Om}|v_k|^{p} dx=0.$$
Hence $v_0\not\equiv 0$ and $\frac{v_0}{\|v_0\|}\in E^{-}_{\la}\cap B^0$,
which is again a contradiction. Thus $u_k$ is bounded. So we may assume that up to a subsequence $u_k\rightharpoonup
u_0$ weakly in $X_0$ and $u_k\ra u_0$ strongly in $L^{p}(\Om)$. Suppose
$u_k\not\ra u_0$ in $X_0$, then
$$\int_{\Om} b(x)|u_0|^{\ba}dx=\lim_{k\ra\infty} \int_{\Om} b(x)|u_k|^{\ba}dx =\left(\frac{1}{p}-\frac{1}{\ba}\right)^{-1}\lim_{k\ra\infty} J_{\la}(u_k)<0$$
and
\begin{align*}
\|u_0\|^p-\la \int_{\Om}|u_0|^{p} dx&
<\lim_{k\ra\infty}\int_{Q} |u_k(x)-u_k(y)|^p K(x-y) dxdy-\la \int_{\Om}|u_k|^{p} dx\\
&=\lim_{k\ra\infty} \int_{\Om} b(x)|u_k|^{\ba} dx = \int_{\Om}
b(x)|u_0|^{\ba} dx.
\end{align*}
Hence $\frac{u_0}{\|u_0\|}\in E^{-}_{\la}\cap B^{-}$ and so
$t(u_0)u_0\in \mc N_{\la}^{-}$, where
\[t(u_0)=\left[\frac{\int_{\Om} b(x)|u_0|^{\ba} dx}{\|u_0\|^{p} - \la \int_{\Om}
|u_0|^{p} dx}\right]^{\frac{1}{p-\ba}}<1.\]
 Moreover, $t(u_0)u_k \rightharpoonup t(u_0)u_0$ weakly in $X_0$ but $t(u_0)u_k
\not\ra t(u_0)u_0$ strongly in $X_0$ and so
\[J_{\la}(t({u_0}) u_{0}) <   \liminf_{k\ra \infty} J_{\la}(t(u_0)u_k).\]
Since the map $t\longmapsto J_{\la}(tu_k)$ attains its maximum at
$t=1$, we have
\[\liminf_{k\ra \infty} J_{\la}(t(u_0)u_k) \leq \lim_{k\ra \infty} J_{\la}(u_k)= \inf_{u\in\mc N_{\la}^{-}}J_{\la}(u).\]
Hence $J_{\la}(t(u_0) u_{0})< \ds\inf_{u\in\mc N_{\la}^{-}}J_{\la}(u)$,
which is impossible. Thus $u_k\ra u_{0}$ strongly in $X_0$, and it
follows easily that $u_0$ is a minimizer for $J_{\la}$ on $\mc
N_{\la}^{-}$.\QED
\end{proof}

\noi In order to prove the existence of non-negative solutions, we first define some
notations.
\[F_{+}=\int_{0}^{t}f_{+}(x,s) ds,\]
where $$ f_{+}(x,t)= \left\{
\begin{array}{lr}
f(x,t)\quad\mbox{if}\quad t\geq0\\
 0\quad\quad\quad\mbox{if}\quad t<0.\\
\end{array}
\right.
$$
In particular, $f(x,t):= b(x) |t|^{\ba-2}t$. Let $J_{\la}^{+}(u)=\|u\|^{p} - \int_{\Om}F_{+}(x,u) dx $.
Then the functional $J_{\la}^{+}(u)$ is well defined and it is
Fre$\acute{c}$het differentiable at $u\in X_0$ and for any $v\in
X_0$
\begin{align}\label{s1}
\ld J_{\la}^{+\prime }(u),v\rd= \int_{Q}
|u(x)-u(y)|^{p-2}(u(x)-u(y))(v(x)-v(y))K(x-y) dxdy -\int_{\Om}
f_{+}(x,u) v dx.
\end{align}
Moreover $J_{\la}^{+}(u)$ satisfies all the above Lemmas and Theorems. So for
$\la\in (0,\la_0)$, there exists two non-trivial critical points
$u_{\la}\in \mc N_{\la}^{+}$ and $v_{\la}\in \mc N_{\la}^{-}$ respectively.

\noi Now we claim that $u_{\la}$ is non-negative in $\mb R^{n}$.
Take $v=u^{-}\in X_0 ($see Lemma $12$ of \cite{ls} in case of $p=2)$, in \eqref{s1}, where $u^{-}=\max(-u,0)$. Then
\begin{align*}
0=& \ld J_{\la}^{+\prime}(u), u^{-}\rd\\
=& \int_{Q} |u(x)-u(y)|^{p-2}(u(x)-u(y))(u^{-}(x)-u^{-}(y))K(x-y)
dxdy -\int_{\Om} f_{+}(x,u) u^{-}(x) dx\\
=&\int_{Q} |u(x)-u(y)|^{p-2}(u(x)-u(y))(u^{-}(x)-u^{-}(y))K(x-y) dxdy\\
=&\int_{Q} |u(x)-u(y)|^{p-2}((u^{-}(x)-u^{-}(y))^2 + 2 u^{-}(x)u^{+}(y))K(x-y) dxdy\\
\geq& \int_{Q} |u^{-}(x)-u^{-}(y)|^{p}K(x-y) dxdy\\
=&\|u^{-}\|^{p}
\end{align*}
Thus $\|u^{-}\|=0$ and hence $u= u^{+}$. So by taking
$u=u_{\la}$ and $u=v_{\la}$ respectively, we get the non-negative
solutions of \eqref{eq01}.


\noi Next we study the asymptotic behavior of the minimizers on $\mc N_{\la}^{+}$ as $\la\ra\la_{1}^{-}$.
\begin{Theorem}\label{t4}
Suppose $\int_{\Om}b(x)\phi_{1}^{\ba} dx >0$. Then
$\ds\lim_{\la\ra\la_{1}^{-}} \inf_{u\in \mc N_{\la}^{+}} J_{\la}(u)=
-\infty$.
\end{Theorem}

\begin{proof}
Clearly we have $\phi_1\in E^{+}_{\la}\cap B^{+}$ for all $\la<\la_1$ and
hence $t(\phi_1)\phi_1\in \mc N^{+}_{\la}$. Now
\begin{align*}
J_{\la}(t(\phi_1)\phi_1)=& \left(\frac{1}{p}-\frac{1}{\ba}\right)
|t(\phi_1)|^{p}\left(\int_{Q}|\phi_1(x)- \phi_1(y)|^p K(x-y)dxdy -
\la \int_{\Om} |\phi_1|^p dx\right)\\
=&\left(\frac{1}{p}-\frac{1}{\ba}\right)
\frac{(\int_{\Om}b(x)|\phi_1|^{\ba}dx)^{\frac{p}{p-\ba}}}{\left(\int_{Q}|\phi_1(x)-
\phi_1(y)|^p K(x-y)dxdy - \la \int_{\Om}|\phi_1|^p
dx\right)^{\frac{\ba}{p-\ba}}}\\
=&\left(\frac{1}{p}-\frac{1}{\ba}\right)
\frac{1}{(\la_1-\la)^{\frac{\ba}{p-\ba}}}\frac{(\int_{\Om}b(x)|\phi_1|^{\ba}dx)^{\frac{p}{p-\ba}}}{\left(
\int_{\Om}|\phi_1|^p dx\right)^{\frac{\ba}{p-\ba}}}
\end{align*}
 Then $\ds\inf_{u\in\mc N_{\la}^{+}} J_{\la}(u)\leq J_{\la}(t(\phi_1)\phi_1)\ra -\infty$ as $\la\ra \la_{1}^{-}$. Hence the result. \QED
\end{proof}

\begin{Corollary}\label{c1}
Let $\int_{\Om} b(x)\phi_{1}^{\ba}dx >0$. Then for every $\la<\la_1$,
there exists a minimizer $u_{\la}$ on $\mc N_{\la}^{+}$ such that
$\ds\lim_{\la\ra\la_1^{-}} \|u_{\la}\|=\infty$.
\end{Corollary}

\noi {\bf Proof of Theorem \ref{sb1}:} Theorem \ref{sb1} follows easily from Theorem \ref{t2}, \ref{t3} and Corollary \ref{c1}.\\

\noi Now we discuss the $p-$sublinear problem with $\la>\la_{1}^{+}$ and $\int_{\Om} b(x)\phi_{1}^{\ba} dx <0$. In this case the hypotheses of Theorem \ref{tt1} hold some way to the right of $\la=\la_1$. More precisely,
\begin{Lemma}\label{les3}
Suppose $\int_{\Om}b(x)\phi_{1}^{\ba} dx <0$. Then there exists
$\de_1, \de_2>0$ such that $u\in E_{\la}^{-}$ implies $\int_{\Om} b|u|^{\ba}dx \leq -\de_2$ whenever
$\la_1< \la \leq \la_1+\de_1$.
\end{Lemma}
 \begin{proof}
We will prove this by a contradiction argument. Suppose there exist sequences
$\{\la_k\}$ and $\{u_k\}$ such that $\|u_k\|=1$, $\la_k\ra
\la_{1}^{+}$ and
\[\int_{Q} |u_k(x)-u_k(y)|^p K(x-y) dxdy-\la_k \int_{\Om}|u_k|^{p} dx< 0 \; \mbox{and}\; \int_{\Om} b(x)|u_k|^{\ba}\ra 0.\]
Since $\{u_k\}$ is bounded, we may assume that $u_k\rightharpoonup
u_0$ weakly in $X_0$ and $u_k\ra u_0$ strongly in $L^{p}(\Om)$ for
$1\leq p<\frac{np}{n-p\al}$. We show that $u_k\ra u_0$ strongly in
$X_0$. Suppose this is not true then
$\|u_0\|<\ds\liminf_{k\ra\infty}\|u_k\|$ and
\[\|u_0\|^p -\la_1 \int_{\Om}|u_0|^{p} dx< \liminf_{k\ra\infty}\left(\|u_k\|^p -\la_k \int_{\Om}|u_k|^{p} dx\right)\leq 0,\]
which is impossible. Hence $u_k\ra u_0$ strongly in $X_0$ and so
$\|u_0\|=1$. It follows that
\[(i) \;\|u_0\|^p -\la_1 \int_{\Om}|u_0|^{p}
dx\leq 0\; \;\;\; (ii)\; \int_{\Om}b(x)|u_{0}|^{\ba} dx = 0.\] But
$(i)$ implies that $u_0= \phi_1$ and then from $(ii)$ we get a contradiction as $\int_{\Om}b(x)\phi_{1}^{\ba} dx<0$. \QED
\end{proof}

\begin{Theorem}\label{t3}
Suppose $\int_{\Om}b(x)\phi_{1}^{\ba} dx <0$ and $\de_1>0$ is as in
Lemma \ref{les3}. Then for  $\la_1< \la \leq \la_1+\de_1$, there
exist minimizers $u_\la$ and $v_\la$ of $J_{\la}$ on $\mc
N_{\la}^{+}$ and $\mc N_{\la}^{-}$ respectively.
\end{Theorem}

\begin{proof}
Clearly $\phi_1\in E^{-}_{\la}$ and so $E^{-}_{\la}$ is non-empty
whenever $\la> \la_1$. By Lemma \ref{les3}, the hypotheses of
Theorem \ref{t2} are satisfied with $\la_0= \la_1+\de_1$ and hence
the result follows.\QED
\end{proof}


\begin{Lemma}\label{t5}
Suppose $\int_{\Om} b(x)\phi_{1}^{\ba}dx <0$, then we have
\begin{enumerate}
\item[(i)] $\lim_{\ds \la\ra \la_{1}^{+}}\inf_{u\in\mc N_{\la}^{-}}J_{\la}(u)= +\infty$.
\item[(ii)] If $\la_k\ra \la_{1}^{+}$ and $u_k$ is a minimizer of $J_{\la_k}$ on $\mc N_{\la}^{-}$, then $\ds\lim_{k\ra\infty} \|u_{k}\|= +\infty$.
\end{enumerate}

\end{Lemma}

\begin{proof}
$(i)$ Let $v\in \mc N_{\la}^{-}$. Then $v=t(u)u$ for some $u\in
E^{-}_{\la}\cap B^{-}$. Now $\int_{\Om} b(x) |u|^{\ba} dx<-\de_2$
provided $\la_1< \la\leq \la_1+\de_1$ and
\[0>\|u\|^p - \la\int_{\Om}|u|^p dx \geq \left(1-\frac{\la}{\la_1}\right)\int_{Q} |u(x)-u(y)|^p K(x-y)dxdy = \frac{\la_1-\la}{\la_1},\]
 so that $|\|u\|^p - \la\int_{\Om}|u|^{p}dx|\leq
 \frac{\la-\la_1}{\la_1}$. Hence
\begin{align*}
J_{\la}(v)&=J_{\la}(t(u)u)=
\left(\frac{1}{p}-\frac{1}{\ba}\right)|t(u)|^p \left(\|u\|^p -
\la\int_{\Om}|u|^p dx \right)\\
&=\left(\frac{1}{\ba}-\frac{1}{p}\right) \frac{|\int_{\Om} b
|u|^{\ba}dx|^{\frac{\ba}{p-\ba}}}{|\|u\|^p-\la_1\int_{\Om}|u|^p
dx|^{\frac{\ba}{p-\ba}}}\geq
\left(\frac{1}{\ba}-\frac{1}{p}\right)\frac{\la_{1}^{\frac{\ba}{p-\ba}}\de_{2}^{\frac{\ba}{p-\ba}}}{(\la-\la_1)^{\frac{\ba}{p-\ba}}}.
\end{align*}
Hence $\ds\inf_{v\in \mc N_{\la}^{-}} J_{\la}(v)\ra + \infty$ as
$\la\ra\la_1^{+}$. This proofs $(i)$.

\noi $(ii)$ is a direct consequence of $(i)$. \QED
\end{proof}
\noi {\bf Proof of Theorem \ref{sb2}:} The proof of Theorem \ref{sb2} follows from Theorem \ref{t3} and Lemma \ref{t5}.

At the end of this section we obtained some non-existence results for $p-$sublinear case.
\begin{Lemma}\label{les6}
Suppose $E_{\la}^{-}\cap B^{+}\ne\emptyset$. Then there exists
$m >0$ such that for every $\e>0$, there exists $u_{\e}\in
E_{\la}^{-}\cap B^{+}$ such that
\[\int_{Q}
|u_{\e}(x)- u_{\e}(y)|^p K(x-y) dxdy - \la \int_{\Om} |u_{\e}|^{p}
dx<\e \;\mbox{and}\; \int_{\Om} b(x)|u_{\e}|^{\ba} dx > m.\]
\end{Lemma}

\begin{proof}
Let $u\in E_{\la}^{-}\cap B^{+}$ then $\|u\|^{p}-\la \int_{\Om}|u|^{p}dx<0$  and $\int_{\Om} b(x)|u|^{\ba}>0$. We may choose $h\in X_0$ with arbitrary small $L^{\infty}$ norm but
$\|h\|^{p}$ is arbitrary large. Thus we may choose $h$ so that $\int_{\Om}b(x)|u+\e h|^{\ba}dx >\frac{1}{2}\int_{\Om}b(x)|u|^{\ba} dx>0$ for $0\leq\e\leq 1$ and $\|u+h\|^{p} -\la\int_{\Om}|u+h|^{p} dx>0$. Let $u_{\e}=\frac{u+\e h}{\|u+\e h\|}$, then we claim that $u_{\e}\in B^{+}$. In fact we have $\frac{1}{\|u+\e h\|^{\ba}}\int_{\Om} b(x)|u+\e h|^{\ba} dx\geq \frac{1}{2(\|u\|+\|h\|)^{\ba}}\int_{\Om}b(x)|u|^{\ba} dx$. Moreover, we have $u_0\in E^{-}_{\la}$ and $u_1\in E^{+}_{\la}$. Let $\eta(\e)= \|u_\e\|^{p}- \la\int_{\Om}|u_{\e}|^{p} dx$ for $0\leq \e \leq 1$. Then $\eta: [0,1]\ra \mb R$ is a continuous function such that $\eta(0)<0$ and $\eta(1)>0$ and so it is easy to see that for any given $\de>0$ there exist $\e$ such that $u_{\e}$ has required properties.\QED

\end{proof}


\begin{Lemma}\label{lea4}
$J_{\la}$ is unbounded below on $\mc N_{\la}$ whenever
$E^{-}_{\la}\cap B^{+}\ne \emptyset$.
\end{Lemma}

\begin{proof}
Let $u\in E^{-}_{\la}\cap B^{+}$. Then by Lemma \ref{les6}, there exists $m>0$ and a sequence $\{u_k\}\subseteq
E^{-}_{\la}\cap B^{+}$ such that $\int_{\Om} b(x) |u_k|^{\ba}dx \geq m$
and $0<\|u_k\|^p-\la_1\int_{\Om}|u_k|^p dx< \frac{1}{k}$. Also
using the same calculation as in Lemma \ref{t5}, we have
\begin{align*}
J_{\la}(t(u_k)u_k)&=\left(\frac{1}{p}-\frac{1}{\ba}\right)
\frac{(\int_{\Om} b(x)
|u_k|^{\ba}dx)^{\frac{p}{p-\ba}}}{(\|u_k\|^p-\la_1\int_{\Om}|u_k|^p
dx)^{\frac{\ba}{p-\ba}}}\\
&<\left(\frac{1}{p}-\frac{1}{\ba}\right)m^{\frac{\ba}{p-\ba}}k^{\frac{p}{p-\ba}}\ra
-\infty\;\mbox{as}\; k\ra\infty.
\end{align*}
Hence $J_{\la}$ is unbounded below on $\mc N_{\la}$.\QED
\end{proof}

\begin{Lemma}\label{le5}
$J_{\la}$ is unbounded below on $\mc N_{\la}$ when either of the
following condition hold:
\begin{enumerate}
\item[$(i)$] $\int_{\Om} b(x)\phi_{1}^\ba dx >0$
and $\la>\la_1$;
\item[$(ii)$] $\la>\la_b$, where $\la_b$ denotes the principal eigenvalue
of
\begin{equation*}
 \quad \left.
\begin{array}{lr}
 \quad -\mc L_{K}u(x) = \la |u(x)|^{p-2}u(x) \; \text{in}\;
\Om^{+} \\
 \quad \quad u = 0 \; \mbox{in}\; \mb R^n \setminus\Om^{+},\\
\end{array}
\quad \right\}
\end{equation*}
\noi with eigenfunction $\phi_b\in X_0$, and $\Om^{+}=\{x\in \Om :
b(x)>0\}$.
\end{enumerate}
\end{Lemma}

\begin{proof}
By Lemma \ref{lea4}, it is sufficient to show that $E^{-}_{\la}\cap
B^{+}\ne \emptyset$. If $(i)$ holds, then $\phi_1\in E^{-}_{\la}\cap
B^{+}$. And if $(ii)$ holds, then $\phi_b\in E^{-}_{\la}\cap B^{+}$.
\QED
\end{proof}

\section{$p-$Superlinear Case $(p<\ba<p^*)$}
In this section we give the proof of Theorems \ref{sp1}, \ref{sp2} and \ref{sp3}. At the end of this section, we also show the non-existence results.
We note that, for $p<\ba<p^*$, it is not difficult to see that
\begin{align}\label{n1}
\mc N_{\la}^{-}= \left\{u\in \mc N_{\la}: \int_{\Om} b(x)|u|^{\ba}dx >0\right\}\;\mbox{and}\; \mc N_{\la}^{+}= \left\{u\in \mc N_{\la}: \int_{\Om} b(x)|u|^{\ba}dx< 0\right\}.
\end{align}

\subsection{Case when $\la<\la_1$}
 When $0<\la<\la_1$, $\int_{Q} |u(x)-u(y)|^p K(x-y) dxdy -
\la \int_{\Om} |u|^{p} dx > (\la_1-\la)\int_{\Om} |u|^p dx >0$ for all $u\in X_0$ and so
$E_{\la}^{+}=\{u\in X_0 : \|u\|=1\}$. Thus $E_{\la}^{-}$ and
$E_{\la}^{0}$ are empty sets and so $\mc N_{\la}^{+}=\emptyset$ and
$\mc N_{\la}^{0}=\{0\}$. Moreover $\mc N_{\la}^{-}=\{t(u)u: u\in
B^{+}\}$ and $\mc N_{\la}=\mc N_{\la}^{-}\cup \{0\}$.

\begin{Lemma}\label{t31}
\begin{enumerate}
\item[$(i)$] If $0<\la<\la_1$ then $J_{\la}(u)$ is bounded below on $\mc N_{\la}^{-}$. And moreover $\ds\inf_{u\in \mc N_{\la}^{-}}J_{\la}(u)>0$.
\item[$(ii)$] There exists a minimizers of $J_{\la}$ on $\mc N_{\la}^{-}$.
\end{enumerate}
\end{Lemma}

\begin{proof}
$(i)$ On $\mc N_{\la}$,
\[J_{\la}(u)= \left(\frac{1}{p} -\frac{1}{\ba}\right)\int_{\Om}
b|u|^{\ba} dx = \left(\frac{1}{p} -\frac{1}{\ba}\right)\left[\|u\|^{p}- \la
\int_{\Om} |u|^{p} dx\right],\]
Thus $J_{\la}(u)\geq 0$ whenever $u\in \mc N_{\la}^{-}$. Hence
$J_\la$ is bounded below by $0$ on $\mc N_{\la}^{-}$. Next, we show
that $\ds\inf_{u\in \mc N_{\la}^{-}}J_{\la}(u)>0$. Suppose $u\in \mc
N_{\la}^{-}$. Then $v=\frac{u}{\|u\|}\in E_{\la}^{+}\cap B^{+}$ and
$u=t(v)v$, where
$t(v)=\left[\frac{\|v\|^p-\la\int_{\Om}|v|^p dx}{\int_{\Om}b|v|^{\ba}}\right]^{\frac{1}{\ba-p}}$. Now
\[\int_{\Om}b(x) |v|^{\ba} dx \leq \overline{b}\int_{\Om}|v|^{\ba}dx \leq  \overline{b} K \|v\|^{\ba/p}= \overline{b} K,\]
where $\overline{b}=\ds\sup_{x\in \Om} b(x)$ and $K$ is a Sobolev embedding constant. Hence
\begin{align*}
J_{\la}(u)= J_{\la}(t(v)v)&= \left(\frac{1}{p}
-\frac{1}{\ba}\right)|t(v)|^{p}\left[\|v\|^{p}- \la \int_{\Om}
|v|^{p} dx\right]\\
&= \left(\frac{1}{p} -\frac{1}{\ba}\right)\frac{\left(\|v\|^{p}- \la
\int_{\Om} |v|^{p}
dx\right)^{\frac{\ba}{\ba-p}}}{\left(\overline{b}K\right)^{\frac{p}{\ba-p}}}
\end{align*}
and hence the result follows.

\noi $(ii)$ Let $\{u_k\}\subseteq\mc N_{\la}^{-}$ be a minimizing
sequence for $J_{\la}$ i.e. $\ds\lim_{k\ra\infty}
J_{\la}(u_k)=\ds\inf_{u\in \mc N_{\la}^{-}}J_{\la}(u)>0$. As
\[J_{\la}(u_k)=  \left(\frac{1}{p} -\frac{1}{\ba}\right)\left[\|u_k\|^{p}- \la
\int_{\Om} |u_k|^{p} dx\right] \geq \left(\frac{1}{p}
-\frac{1}{\ba}\right)\left(1 -\frac{\la}{\la_1}\right)\|u_k\|^p,\]
so $\{u_k\}$ is a bounded sequence in $X_0$. Thus we may assume that up to a
subsequence still denoted by $\{u_k\}$ such that $u_k\rightharpoonup
u_0$ weakly in $X_0$ and $u_k\ra u_0$ in $L^{p}(\Om)$ and
$L^{\ba}(\Om)$. Now $0< \ds\lim_{k\ra\infty}J_{\la}(u_k)=
\left(\frac{1}{p} -\frac{1}{\ba}\right)\lim_{k\ra\infty}\int_{\Om}
b(x)|u_k|^{\ba} dx=\left(\frac{1}{p} -\frac{1}{\ba}\right)\int_{\Om}
b(x)|u_0|^{\ba} dx $. It follows that $u_0\not\equiv 0$ a.e. in $\mb
R^n$. Also $\|u_0\|^{p}- \la \int_{\Om} |u_0|^{p} dx \geq
(\la_1-\la)\int_{\Om}|u_0|^p dx>0$. Thus $\frac{u_0}{\|u_0\|}\in
B^{+}\cap E^{+}_{\la}$. We now show that $u_k\ra u_0$ strongly in
$X_0$. Suppose $u_k\not\ra u_0$ in $X_0$. Then
\begin{align*}
\|u_0\|^p - \la \int_{\Om} |u_0|^{p} dx -\int_{\Om}
b(x)|u_0|^{\ba}dx <\liminf_{k\ra\infty} \|u_k\|^p
- \la \int_{\Om} |u_k|^{p} dx - \int_{\Om} b(x)|u_k|^{\ba} dx =0.
\end{align*}
Also by the fibering map analysis we have that $\phi_{u_0}$ has a
unique maximum at $t(u_0)$ such that $t(u_0)u_0\in \mc N_{\la}^{-}$
and $t(u_0)<1$. As $u_k\in \mc N_{\la}^{-}$, the map $\phi_u$
attains its maximum at $t=1$. Hence
\[J_{\la}(t(u_0) u_{0}) <  \liminf_{k\ra \infty} J_{\la}(t(u_0)u_{k})\leq \lim_{k\ra \infty}J_{\la}(u_k)= \inf_{u\in\mc N_{\la}^{-}}J_{\la}(u).\]
which is a contradiction. Hence we must have $u_k\ra u_{0}$ in
$X_0$. Thus $u_{0}\in \mc N_{\la}^{-}$ and
$J_{\la}(u_0)=\ds\lim_{k\ra\infty} J_{\la}(u_k)=\ds\inf_{u\in\mc
N_{\la}^{-}}J_{\la}(u)$. Since $\int_{\Om} b|u_0|^{\ba} dx>0$,
$u_0\not\in \mc N_{\la}^{0}$. So $u_0$ is a critical point of
$J_{\la}$.\QED
\end{proof}
\begin{Theorem}\label{t41}
Suppose $\int_{\Om}b(x)\phi_{1}^{\ba} dx >0$. Then
\begin{enumerate}
\item[$(i)$] $\lim_{\la\ra\la_{1}^{-}} \inf_{u\in \mc N_{\la}^{-}}
J_{\la}(u)= 0$.
\item[$(ii)$] If $\la_{k}\ra \la_{1}^{-}$ and $u_k$ is a minimizer
of $J_{\la_{k}}$ on $\mc N_{\la}^{-}$, then $\lim_{k\ra\infty}
u_k=0$
\end{enumerate}
\end{Theorem}

\begin{proof}
\begin{enumerate}
\item[$(i)$] Without loss of generality, we may assume that $\|\phi_1\|=1$. Since $\int_{\Om}
b(x)\phi_{1}^{\ba}dx >0$ and
\[\int_{Q} |\phi_{1}(x)-\phi_{1}(y)|^p K(x-y)dxdy - \la\int_{\Om}|\phi_1|^{p}dx = (\la_1- \la)\int_{\Om}|\phi_1|^{p}dx >0,\]
we have $\phi_1\in E^{+}_{\la}\cap B^{+}$ for all $\la<\la_1$ and
hence $t(\phi_1)\phi_1\in \mc N^{-}_{\la}$, where
$t(\phi_1)=
\left[\frac{(\la_1-\la)\int_{\Om}|\phi_1|^p
dx}{\int_{\Om}b(x)|\phi_1|^{\ba}dx}\right]^{\frac{1}{\ba-p}}$. Thus
\begin{align*}
J_{\la}(t(\phi_1)\phi_1)=& \left(\frac{1}{p}-\frac{1}{\ba}\right)
|t(\phi_1)|^{\ba} \int_{\Om}b(x) |\phi_1|^{\ba} dx\\
=&\left(\frac{1}{p}-\frac{1}{\ba}\right)
(\la_1-\la)^{\frac{\ba}{\ba-p}}\frac{(\int_{\Om}|\phi_1|^{p}dx)^{\frac{\ba}{\ba-p}}}{\left(
\int_{\Om}b(x)|\phi_1|^{\ba} dx\right)^{\frac{p}{\ba-p}}}.
\end{align*}
Then $0<\ds\inf_{u\in\mc N_{\la}^{-}} J_{\la}(u)\leq
J_{\la}(t(\phi_1)\phi_1)\ra 0$ as $\la\ra \la_{1}^{-}$. Hence $\ds\lim_{\la\ra\la_{1}^{-}} \inf_{u\in \mc N_{\la}^{-}}
J_{\la}(u)= 0$.

\item[$(ii)$] We first show that $\{u_k\}$ is bounded. Suppose not, then we
may assume that $\|u_k\|\ra\infty$ as $k\ra \infty$. Let
$v_k=\frac{u_k}{\|u_k\|}$. Then we may assume that
$v_k\rightharpoonup v_0$ weakly in $X_0$ and $v_k\ra v_0$
strongly in $L^{p}(\Om)$ for every $1\leq p<p^*$. Since $u_k\in \mc
N_{\la}$, we have
\begin{align*}
J_{\la_k}(u_k)= \left(\frac{1}{p}-\frac{1}{\ba}\right)\left[ \|u_k\|^p
- \la_k \int_{\Om} |u_k|^{p} dx\right] =
\left(\frac{1}{p}-\frac{1}{\ba}\right)\int_{\Om}
b(x)|u_k|^{\ba}dx\ra 0\;\mbox{as}\;k\ra\infty,
\end{align*}
by $(i)$ and so we get
\[\lim_{k\ra\infty}\left(\|v_k\|^p  - \la_k \int_{\Om} |v_k|^{p}
dx\right) = 0\;\mbox{and}\; \lim_{k\ra\infty} \int_{\Om}
b(x)|v_k|^{\ba} dx=0.\] \noi Suppose $v_k\not\ra v_0$ strongly in
$X_0$. Then
\[\|v_0\|^p -\la_1\int_{\Om}|v_0|^{p}dx < \lim_{k\ra\infty} \int_{Q} |v_k(x)-v_k(y)|^p K(x-y) dxdy-\la_k\int_{\Om}|v_k|^{p}dx =0,\]
which is impossible. Hence $v_k\ra v_0$ in $X_0$. Thus we must have
\[\|v_0\|^p -\la_1\int_{\Om}|v_0|^{p}dx= \lim_{k\ra\infty}  \|v_k\|^p - \la_k\int_{\Om}|v_k|^{p}dx = 0,\]
and so $v_0=k\phi_1$ for some $k$. Since $\int_{\Om}
b(x)|v_0|^{\ba}dx=0$ implies that $k=0$. Thus $v_0=0$, which is
again impossible as $\|v_0\|=1$. Hence $\{u_k\}$ is bounded. So we
assume that $u_k\rightharpoonup u_0$ weakly in $X_0$. Thus by using
the same argument, we can get that $u_k\ra u_0$ and $u_0=0$. Hence
the proof is complete. \QED
\end{enumerate}
\end{proof}
\noi {\bf Proof of Theorem 1.4:} Lemma \ref{t31} and Theorem \ref{t41} complete the proof of Theorem \ref{sp1}.
\subsection{Case when $\la>\la_1$}
If $\la>\la_1$, then
$$\int_{Q} |\phi_{1}(x)-\phi_{1}(y)|^p K(x-y) dxdy - \la
\int_{\Om} |\phi_{1}|^{p} dx = (\la_1 -\la) \int_{\Om}
|\phi_{1}|^{p} dx <0.$$ and so $\phi_1\in E^{-}_{\la}$. Hence if $\int_{\Om} b(x)|\phi_1|^{\ba}<0$ then $\phi_1\in E_{\la}^{-}\cap B^{-}$ and so $\mc N_{\la}^{+}$ is non-empty. For $\la=\la_1$, we have $E_{\la}^{-}=\emptyset$ and $E_{\la}^{0}=\{\phi_1\}$.\\
When $\la>\la_1$, and if $\phi_1\in B^{-}$, then it follows that $\overline{E^{-}_{\la}}\cap\overline{B^{+}}$ is empty. We show that this is an important condition for establishing the existence of minimizers.

\begin{Lemma}\label{le3}
Suppose $\int_{\Om}b(x)\phi_{1}^{\ba} dx <0$ then there exists
$\de>0$ such that $u\in \overline{E_{\la}^{-}}\cap
\overline{B^{+}}=\emptyset$ whenever $\la_1< \la \leq \la_1+\de$.
\end{Lemma}
 \begin{proof}
 This can be prove in a similar way as in Lemma \ref{les3}.\QED
\end{proof}
\begin{Theorem}\label{t1}
Suppose $\overline{E^{-}_{\la}}\cap\overline{B^{+}}=\emptyset$. Then
we have the following:
\begin{enumerate}
\item[1.]  ${\mc N_{\la}^{0}}= \{0\}$.
\item[2.] $0\not\in \overline{\mc N_{\la}^{-}}$ and $\mc N_{\la}^{-}$ is closed.
\item[3.] $\mc N_{\la}^{-}$ and ${\mc N_{\la}^{+}}$ are separated, i.e. $\overline{\mc N_{\la}^{-}}\cap \oline{\mc N_{\la}^{+}}=\emptyset$.
\item[4.]  $\mc N_{\la}^{+}$ is bounded.
\end{enumerate}
\end{Theorem}

\begin{proof}
\begin{enumerate}
\item[$1.$] Let $u_0\in \mc N_{\la}^{0}\setminus \{0\}$. Then
$\frac{u_0}{\|u_0\|}\in E_{\la}^{0}\cap B^{0}\subseteq
\overline{E_{\la}^{-}}\cap \overline{B^{+}}=\emptyset$, which is
impossible. Hence ${\mc N_{\la}^{0}}= \{0\}$.

\item[$2.$] Suppose by contradiction that $0\in\overline{\mc N_{\la}^{-}}$. Then there exists
a sequence $\{u_k\}\subseteq \mc N_{\la}^{-}$ such that $\lim_{k\ra\infty} u_k
=0$ in $X_0$. Since $u_k\in \mc N_{\la}$,
\[0<\|u_k\|^p - \la \int_{\Om} |u_k|^{p} dx = \int_{\Om} b(x)|u_k|^{\ba}dx \ra 0\;\mbox{as}\; k\ra\infty\] implies that
\[\lim_{k\ra\infty}\int_{\Om} b(x)|u_k|^{\ba}dx=0\;\mbox{and}\;\lim_{k\ra\infty}\left(\|u_k\|^p - \la \int_{\Om} |u_k|^{p} dx\right)=0.\]
Let $v_k=\frac{u_k}{\|u_k\|}$. Then up to a subsequence
$v_k\rightharpoonup v_0$ weakly in $X_0$ and $v_k\ra v_0$ strongly in $L^{p}(\Om)$.
Clearly
\begin{align*}
0<\|v_k\|^p - \la \int_{\Om} |v_k|^{p} dx &=
{\|u_k\|^{\ba-p}}\int_{\Om} b(x)|v_k|^{\ba}dx\ra
0\;\mbox{as}\;k\ra\infty.
\end{align*}
Thus we have
\[0= \lim_{k\ra\infty} \left(\|v_k\|^p - \la \int_{\Om} |v_k|^{p} dx \right)= 1-\la\int_{\Om}|v_0|^{p}dx\]
and so $v_0\ne 0$. Moreover
\[\|v_0\|^p - \la \int_{\Om} |v_0|^{p} dx \leq \lim_{k\ra\infty} \|v_k\|^p - \la \int_{\Om} |v_k|^{p} dx =0,\]
and so $\frac{v_0}{\|v_0\|}\in \overline{E_{\la}^{-}}$. Since
$\int_{\Om} b(x)|v_k|^{\ba} dx>0$, it follows that $\int_{\Om}
b(x)|v_0|^{\ba} dx\geq0$ and so $\frac{v_0}{\|v_0\|}\in
\overline{B^+}$, which is a contradiction. Thus we have $0\not\in
\mc N_{\la}^{-}$.

\noi We now show that $\mc N_{\la}^{-}$ is a closed set.
Clearly $\overline{\mc N_{\la}^{-}}\subseteq\mc N_{\la}^{-}\cup \{0\}$. But
$0\not\in \mc N_{\la}^{-}$ so it follows that $\overline{\mc
N_{\la}^{-}}= \mc N_{\la}^{-}$.

\item[$3.$]  Using $(i)$ and $(ii)$, we have $\overline{\mc N_{\la}^{-}}\cap\overline{\mc
N_{\la}^{+}}\subseteq \mc N_{\la}^{-}\cap(\mc N_{\la}^{+}\cup \mc
N_{\la}^{0})= (\mc N_{\la}^{-}\cap \mc N_{\la}^{+})\cup(\mc
N_{\la}^{-}\cap \{0\})= \emptyset$, and so $\mc N_{\la}^{-}$ and
$\mc N_{\la}^{+}$ are separated.

\item[$4.$] Suppose $\mc N_{\la}^{+}$ is not bounded. Then as in Theorem there exists a
sequence $\{u_k\}\subseteq \mc N_{\la}^{+}$ and $v_k=\frac{u_k}{\|u_k\|}$ satisfy
$\|u_k\|\ra\infty$ as $k\ra \infty$ and $\|u_k\|^p - \la \int_{\Om} |u_k|^{p} dx
= \int_{\Om} b(x)|u_k|^{\ba}dx<0$
and
\[\|v_k\|^p  - \la \int_{\Om} |v_k|^{p} dx
= {\|u_k\|^{\ba-p}}\int_{\Om} b(x)|v_k|^{\ba}dx.\]
Since $\|v_k\|^p -\la \int_{\Om} |v_k|^{p} dx$ is bounded and $\|u_k\|
\ra\infty$ as $k\ra \infty$, we have $\int_{\Om} b(x)|v_0|^{\ba}dx=
\ds\lim_{k\ra\infty}\int_{\Om} b(x)|v_k|^{\ba}dx=0.$
\noi We now show that $v_k\ra v_0$ strongly in $X_0$. Suppose $v_k\not\ra v_0$ strongly in $X_0$. Then from \eqref{n1},
\begin{align}\label{n2}
\|v_0\|^p -\la\int_{\Om}|v_0|^{p}<\lim_{k\ra\infty} \int_{Q} |v_k(x)-v_k(y)|^p K(x-y) dxdy-\la\int_{\Om}|v_k|^{p} \leq0.
\end{align}
Thus $\frac{v_0}{\|v_0\|}\in \overline{E^{-}_{\la}}\cap
\overline{B^{+}}$, which is a contradiction. Hence $v_k\ra v_0$ in $X_0$.
Therefore $\|v_0\|=1$. From this and equation \eqref{n2} we obtain
$v_0\in \overline{E^{-}_{\la}}\cap\overline{B^{+}}$, which is
again a contradiction. Hence $\mc N^{+}_{\la}$ is bounded.\QED
\end{enumerate}
\end{proof}
Next we show that $J_{\la}$ is bounded below on $\mc
N_{\la}^{+}$ and bounded away from zero on $\mc N_{\la}^{-}$.
Moreover for $\la<\la_0$, $J_{\la}$ achieves its minimizers on $\mc
N_{\la}^{+}$ and $\mc N_{\la}^{-}$ provided $\mc N_{\la}^{-}$ is
non-empty. We also note that $J_{\la}(u)$ changes sign in $\mc N_{\la}$. We
have $J_{\la}(u)>0$ on $\mc N_{\la}^{-}$ and $J_{\la}(u)<0$ on $\mc
N_{\la}^{+}$.



\begin{Theorem}\label{tt2}
Suppose $\oline{E^{-}_{\la}}\cap \oline{B^{+}}=\emptyset$, Then, we have the following \\
$(i)$ every minimizing sequence of $J_{\la}(u)$ on $\mc N_{\la}^{-}$ is bounded.\\
$(ii)$ $\ds\inf_{u\in \mc N_{\la}^{-}}J_{\la}(u)>0$.\\
$(iii)$ there exists a minimizer for $J_{\la}(u)$ on $\mc N_{\la}^{-}$.
\end{Theorem}

\begin{proof}
$(i)$ Let $\{u_k\}\in \mc N_{\la}^{-}$ be a minimizing sequence for $J_{\la}$ on $\mc
N_{\la}^{-}$. Then
\[\|u_k\|^p - \la \int_{\Om}|u_k|^{p} dx=\int_{\Om} b(x)|u_k|^{\ba} dx\ra c\geq 0\]
We claim that $\{u_k\}$ is a bounded sequence. Suppose this is not true i.e $\|u_k\|\ra \infty$ as
$k\ra\infty$. Let $v_k=\frac{u_k}{\|u_k\|}$. Then  $v_k\rightharpoonup v_0$ weakly in $X_0$ and $v_k\ra v_0$ strongly in $L^{p}(\Om)$. Also
\[\lim_{k\ra\infty}\int_{Q} |v_k(x)- v_k(y)|^p K(x-y) dxdy -\la \int_{\Om} |v_k|^{p} dx =\lim_{k\ra\infty} \int_{\Om}b(x)|v_k|^{\ba} \|u_k\|^{\ba-p}dx\ra 0.\]
Since $\|u_k\|\ra+\infty$, it follows that $\int_{\Om}b(x)|v_k|^{\ba}dx\ra 0$ as $k\ra\infty$ and so
$\int_{\Om}b(x)|v_0|^{\ba}dx =0$.
 Next, suppose $v_k\not\ra v_0$ in $X_0$ and so
\[\|v_0\|^p -\la \int_{\Om}|v_0|^{p} dx<\lim_{k\ra\infty}\|v_k\|^p-\la \int_{\Om}|v_k|^{p} dx=0.\]
Thus $v_0\ne 0$ and $\frac{v_0}{\|v_0\|}\in \oline{E_{\la}^{-}}\cap \oline{B^{+}}$ which is impossible. Hence $v_k\ra v_0$ strongly in $X_0$.
It follows that $\|v_0\|=1$ and $\|v_0\|^p -\la \int_{\Om}|v_0|^{p} dx = \int_{\Om} b(x)|v_0|^{\ba}=0$.
Thus, $\frac{v_0}{\|v_0\|}\in E^{0}_{\la}\cap B^0$,
which is again a contradiction as $\oline{E^{-}_{\la}}\cap \oline{B^+}=\emptyset$. Hence $\{u_k\}$ is bounded.\\

$\noi (ii)$ Clearly $\ds\inf_{u\in \mc N_{\la}^{-}}J_{\la}(u)\geq0$. Suppose $\ds\inf_{u\in \mc N_{\la}^{-}}J_{\la}(u)=0$. Then let
$\{u_k\}$ be a minimizing sequence such that $J_{\la}(u_k)\ra 0$. By $(i)$, $\{u_k\}$ is bounded. Thus we may assume that
$u_k\rightharpoonup u_0$ weakly in $X_0$ and $u_k\ra u_0$ in $L^{p}(\Om)$.
Also $u_k\in \mc N_{\la}^{-}$ implies that $\int_{\Om} b(x)|u_0|^{\ba}dx\geq 0$. Now suppose
$u_k\not\ra u_0$ in $X_0$ then
\[\|u_0\|^p -\la \int_{\Om} |u_0|^{p} dx < \lim_{k\ra\infty}\int_{Q} |u_k(x)- u_k(y)|^p K(x-y) dxdy -\la \int_{\Om} |u_k|^{p} dx=0\]
which implies that $\frac{u_0}{\|u_0\|}\in \oline{E^{-}_{\la}}\cap \oline{B^+}$, which is impossible. Hence $u_k\ra u_0$. Also  $u_0\ne 0$,
since $0\not\in\oline{\mc N_{\la}^{-}}$. It then follows exactly as in the proof in $(i)$ that $\frac{u_0}{\|u_0\|}\in E^{0}_{\la}\cap B^{0}$ which is impossible as $\oline{E^{-}_{\la}}\cap \oline{B^+}=\emptyset$.\\

\noi $(iii)$ Let $\{u_k\}$ be a minimizing sequence. Then
\[J_{\la}(u_k)= \left(\frac{1}{p}-\frac{1}{\ba}\right)(\|u_k\|^{p} -\la \int_{\Om}|u_k|^{p} dx)= \left(\frac{1}{p}-\frac{1}{\ba}\right)\int_{\Om} b(x)|u_k|^{\ba}dx\ra \inf_{u\in \mc N_{\la}^{-}}J_{\la}(u)>0.\]
Also by $(i)$, $\{u_k\}$ is bounded. Therefore, we may assume that $u_k \rightharpoonup u_0$ weakly in $X_0$ and $u_k\ra u_0$ strongly in $L^{p}$.
Then $\int_{\Om} b(x)|u_0|^{\ba}dx>0$. Since $\oline{E^{-}_{\la}}\cap \oline{B^+}=\emptyset,$ it follows that
$B^{+}\subseteq E^{+}_{\la}$ and so $\|u_0\|^p-\la \int_{\Om}|u_0|^{p} dx>0.$
Hence $\frac{u_0}{\|u_0\|}\in E^{+}_{\la}\cap B^{+}$. Therefore there exists $t(u_0)$ such that
$t(u_0)u_0\in \mc N_{\la}^{-}$, where
\[t(u_0)=\left[\frac{\|u_0\|^{p} - \la \int_{\Om}
|u_0|^{p} dx}{\int_{\Om} b(x)|u_0|^{\ba} dx}\right]^{\frac{1}{\ba-p}}.\]
We now show that $u_k\ra u_0$ strongly in $X_0$. Suppose not, then $$\|u_0\|^{p} - \la \int_{\Om}
|u_0|^{p} dx< \lim_{k\ra\infty}\|u_k\|^{p} - \la \int_{\Om}
|u_k|^{p} dx =  \lim_{k\ra\infty}\int_{\Om} b(x)|u_k|^{\ba} dx= \int_{\Om} b(x)|u_0|^{\ba} dx$$
 and so $t(u_0)<1$. Since $t(u_0)u_k \rightharpoonup t(u_0)u_0$ weakly in $X_0$ but $t(u_0)u_k
\not\ra t(u_0)u_0$ strongly in $X_0$ and so
\[J_{\la}(t({u_0}) u_{0}) <   \lim_{k\ra \infty} J_{\la}(t(u_0)u_k).\]
Since the map $t\mapsto J_{\la}(tu_k)$ attains its maximum at
$t=1$, we have
\[J_{\la}(t(u_0)u_0)<\liminf_{k\ra \infty} J_{\la}(t(u_0)u_k) \leq \lim_{k\ra \infty} J_{\la}(u_k)= \inf_{u\in\mc N_{\la}^{-}}J_{\la}(u),\]
which is impossible. Thus $u_k\ra u_{0}$ strongly in $X_0$, and it
follows easily that $u_0$ is a minimizer for $J_{\la}$ on $\mc
N_{\la}^{-}$.\QED
\end{proof}


\begin{Theorem}\label{ttt1}
Suppose $E_{\la}^{-}$ is non-empty but $\oline{E^{-}_{\la}}\cap \oline{B^{+}}=\emptyset$. Then there exist a minimizer of $J_{\la}$ on $\mc N_{\la}^{+}$.
\end{Theorem}

\begin{proof}
Since $\oline{E^{-}_{\la}}\cap \oline{B^{+}}=\emptyset$, ${E^{-}_{\la}}\cap {B^{-}}\not=\emptyset$ and so $\mc N_{\la}^{+}$ must be nonempty.
Also by Theorem \ref{t1}, we have $\mc N_{\la}^{+}$ is bounded so there exist $M>0$ such that $\|u\|\leq M$ for all $u\in \mc N_{\la}^{+}$. Hence by using Sobolev inequality, we have
\begin{align*}
J_{\la}(u)= &\left(\frac{1}{p} -\frac{1}{\ba}\right)\int_{\Om}
b(x)|u|^{\ba} dx \geq \left(\frac{1}{p} -\frac{1}{\ba}\right)\underline{b}\int_{\Om}
|u|^{\ba} dx\\
&\geq \left(\frac{1}{p} -\frac{1}{\ba}\right)\underline{b} K \|u\|^{\ba}\geq \left(\frac{1}{p} -\frac{1}{\ba}\right)\underline{b} K M^{\ba}
\end{align*}
where $\underline{b}=\ds\inf_{x\in \Om} b(x)$. Thus $J_{\la}$ is bounded below on $\mc N_{\la}^{+}$ and so $\ds\inf_{u\in \mc N_{\la}^{+}}J_{\la}(u)$ exists. Moreover, $\ds\inf_{u\in \mc N_{\la}^{+}}J_{\la}(u)<0$.

\noi Suppose that $\{u_k\}$ is a minimizing sequence on $\mc N_{\la}^{+}$. Then
\[J_{\la}(u_k)= \left(\frac{1}{p} -\frac{1}{\ba}\right)\left[\|u_k\|^{p}- \la
\int_{\Om} |u_k|^{p} dx \right]= \left(\frac{1}{p} -\frac{1}{\ba}\right)\int_{\Om}
b|u_k|^{\ba} dx \ra \inf_{u\in \mc N_{\la}^{+}}J_{\la}(u) <0\]
as $k\ra\infty$. Since $\mc N_{\la}^{+}$ is bounded, we may assume that $u_k\rightharpoonup
u_0$ weakly in $X_0$ and $u_k\ra u_0$ in $L^{p}(\Om)$ and
$L^{\ba}(\Om)$. Then
\[\int_{\Om}
b|u_0|^{\ba} dx = \lim_{k\ra\infty}\int_{\Om}
b|u_k|^{\ba} dx <0\;\mbox{and}\; \|u_0\|^{p}- \la
\int_{\Om} |u_0|^{p} dx <\lim_{k\ra\infty}\left[\|u_k\|^{p}- \la
\int_{\Om} |u_k|^{p} \right]<0.\]
Hence $\frac{u_0}{\|u_0\|}\in E_{\la}^{-}\cap B^{-}$ and so there exist $t(u_0)$ such that $t(u_0)u_0\in \mc N_{\la}^{+}$.
Suppose $u_k\not\ra u_0$ then
\[\|u_0\|^{p}- \la
\int_{\Om} |u_0|^{p} dx <\lim_{k\ra\infty}\left[\|u_k\|^{p}- \la
\int_{\Om} |u_k|^{p}dx \right]= \lim_{k\ra\infty} \int_{\Om}
b|u_k|^{\ba} dx = \int_{\Om}
b|u_0|^{\ba} dx <0.\]
So \[t(u_0)=\left[\frac{\|u_0\|^{p} - \la \int_{\Om}
|u_0|^{p} dx}{\int_{\Om} b(x)|u_0|^{\ba} dx}\right]^{\frac{1}{\ba-p}}>1.\]
But this leads to a contradiction as
\[J_{\la}(t(u_0) u_{0}) < J_{\la}(u_0)\leq \lim_{k\ra \infty} J_{\la}(u_{k})=\inf_{u\in\mc N_{\la}^{+}}J_{\la}(u).\]
Thus we must have $u_k\ra u_{0}$ in
$X_0$, and so $\|u_0\|^p-\la\int_{\Om}|u_0|^p dx =\int_{\Om} b|u_0|^{\ba} dx<0$. Thus $u_{0}\in \mc N_{\la}^{+}$ and
$J_{\la}(u_0)=\ds\lim_{k\ra\infty} J_{\la}(u_k)=\ds\inf_{u\in\mc
N_{\la}^{+}}J_{\la}(u)$. Since $\int_{\Om} b|u_0|^{\ba} dx<0$,
$u_0\not\in \mc N_{\la}^{0}$ and so $u_0$ is a critical point of
$J_{\la}$.\QED
\end{proof}

\begin{Theorem}\label{tt3}
Suppose $\int_{\Om}b(x)\phi_{1}^{\ba} dx <0$. Then there exists $\de_1>0$ such that for $\la_1< \la \leq \la_1+\de_1$ there
exist minimizers $u_\la$ and $v_\la$ of $J_{\la}$ on $\mc
N_{\la}^{+}$ and $\mc N_{\la}^{-}$ respectively.
\end{Theorem}

\begin{proof}
Clearly $\phi_1\in E^{-}_{\la}$ and so $E^{-}_{\la}$ is non-empty
whenever $\la\geq \la_1$. By Lemma \ref{le3}, the hypotheses of
Theorem \ref{tt2} and Theorem \ref{ttt1} are satisfied with $\la_0= \la_1+\de_1$ and hence
the result follows.\QED
\end{proof}

\noi By considering $J_{\la}^{+}$ as in $p-$sublinear case, we get non-negative solutions in the similar way.
\noi Finally, in this section we investigate the behavior of $\mc N_{\la}^{+}$ as $\la\ra \la_{1}^{+}$

 \begin{Theorem}\label{tt4}
Suppose $\int_{\Om} b(x)\phi_{1}^{\ba}dx <0$ and $u_k\in \mc N_{\la}^{+}$ for $\la=\la_k$ where $\la_{k}\ra \la_{1}^{+}$.
Then as $k\ra \infty$ we have $(i)$ $u_k\ra 0$ and $(ii)$ $\frac{u_k}{\|u_k\|}\ra \phi_1$ in $X_0$.
 \end{Theorem}

 \begin{proof}
$(i)$ As $\mc N_{\la}^{+}$ is bounded so we may suppose that $u_k\rightharpoonup u_0$ weakly in $X_0$ and $u_k\ra u_0$ in $L^{p}(\Om)$.
Also \[\|u_k\|^p  -\la_k \int_{\Om} |u_k|^{p} dx=  \int_{\Om} b(x)|u_k|^{\ba}dx<0\;\mbox{for all}\;k.\]
 Now suppose that $u_k\not\ra u_0$ in $X_0$ then
\[\|u_0\|^p -\la_1 \int_{\Om} |u_0|^{p} dx < \liminf_{k\ra\infty}\left[\int_{Q} |u_k(x)- u_k(y)|^p K(x-y) dxdy -\la_k \int_{\Om} |u_k|^{p}dx\right] \leq 0\]
which is impossible. Hence $u_k\ra u_0$ strongly in $X_0$ and so
\[\|u_0\|^p  -\la_1 \int_{\Om} |u_0|^{p} dx=  \int_{\Om} b(x)|u_0|^{\ba}dx \leq 0.\]
Hence $\|u_0\|^p  -\la_1 \int_{\Om} |u_0|^{p} dx= 0$ so $u_0=k\phi_1$ for some $k$. But as $\int_{\Om} b(x)\phi_{1}^{\ba}dx< 0$ we obtain $k=0$.
Thus $u_k\ra 0$ in $X_0$.

\noi $(ii)$ Let $v_k=\frac{u_k}{\|u_k\|}$. Then we may assume that $v_k\rightharpoonup v_0$ weakly in $X_0$ and $v_k\ra v_0$ in $L^{p}(\Om)$.
Clearly \[\|v_k\|^p  -\la_k \int_{\Om} |v_k|^{p} dx=  \int_{\Om} b(x)|v_k|^{\ba}\|u_k\|^{\ba-p}dx.\]
Since $\|u_k\|\ra 0$ as $k\ra\infty$, we have $\lim_{k\ra\infty}\|v_k\|^p  -\la_1 \int_{\Om} |v_k|^{p} dx= 0.$
 We claim that $v_k\ra v_0$ strongly in $X_0$. Suppose not, then
\[\|v_0\|^p -\la_1 \int_{\Om} |v_0|^{p} dx < \lim_{k\ra\infty}\left[\int_{Q} |v_k(x)- v_k(y)|^p K(x-y) dxdy -\la_1 \int_{\Om} |v_k|^{p}dx\right] \leq 0\]
 which gives a contradiction. Hence $v_k\ra v_0$ strongly in $X_0$ and so $\|v_0\|=1$ and $\|v_0\|^p -\la_1 \int_{\Om} |v_0|^{p} dx=0$. Thus
$v_0=\phi_1$ and hence the result.\QED
\end{proof}

\noi {\bf Proof of Theorem \ref{sp2}:} It follows from Theorem \ref{tt3} and \ref{tt4}.\\


\noi At the end, we study non-existence results in $p-$superlinear case. For this, if $\int_{\Om} b(x)\phi_{1}^{\ba}dx >0$ then $\phi_1\in E_{\la}^{-}\cap B^{+}$ whenever $\la>\la_1$. One can easily show in a similar way as in Lemma \ref{le3}  that there exists $\de>0$ such that $\oline{E_{\la}^{-}}\subset B^{+}$, whenever $\la_1\leq \la<\la+\de$.
i.e $E_{\la}^{-}\cap B^{-}=\phi$ and so $\mc N_{\la}^{+}$ is empty. On the other hand $\mc N_{\la}^{-}$ is non-empty but we have

\begin{Lemma}\label{ttt2}
If $E_{\la}^{-}\cap B^{+}\not= \emptyset$, then $\inf_{u\in \mc N_{\la}^{-}}J_{\la}(u)=0$.
\end{Lemma}

 \begin{proof}
Let $u\in E_{\la}^{-}\cap B^{+}$ then it is possible to choose $h\in X_0$ with sufficiently small $L^{\infty}$ norm but sufficiently large $X$ norm
so that $\|u+\e h\|^p -\la \int_{\Om} |u+\e h|^{p} dx>0$ and $ \int_{\Om} b(x)|u+\e h|^{\ba} dx> \frac{1}{2} \int_{\Om} b(x) |u|^{\ba} dx$ for any $0\leq \e\leq1$.
Let $v_{\e}=\frac{u+\e h}{\|u+\e h\|}$ then $v_0\in E_{\la}^{-}$, $v_1\in E_{\la}^{+}$ and there exists $\e_0\in(0,1)$ such that
$v_{\e_0}\in E_{\la}^{0}$. Moreover, there exists a sequence $\{v_k\}\in E_{\la}^{+}\cap B^{+}$ $(v_k=v_{\e_k})$ such that $\ds\lim_{k\ra \infty}\left[
\|v_k\|^p  -\la \int_{\Om} |v_k|^{p} dx \right]=0$ and
\[\int_{\Om} b(x)|v_k|^{\ba}dx= \frac{1}{\|u + \e_{k} h\|^{\ba}}\int_{\Om} b(x)|u+\e_{k} h|^{\ba}dx\geq  \frac{1}{2(\|u\|+\| h\|)^{\ba}}\int_{\Om} b(x)|u|^{\ba}dx.\]
Hence \[\lim_{k\ra\infty} t(v_k)=\lim_{k\ra\infty} \left[\frac{\|v_k\|^{p} - \la \int_{\Om}
|v_k|^{p} dx}{\int_{\Om} b(x)|v_k|^{\ba} dx}\right]^{\frac{1}{\ba-p}}=0.\]
Now $t(v_k)v_k\in \mc N_{\la}^{-}$, we have
\[J_{\la}(t(v_k)v_k)= \left(\frac{1}{p}-\frac{1}{\ba}\right)|t(v_k)|^{\ba} \int_{\Om} b(x)|v_k|^{\ba} dx \ra 0.\]
Hence $\ds\inf_{u\in \mc N_{\la}^{-}} J_{\la}(u)=0$. \QED
\end{proof}
\begin{Corollary}\label{c2}
If $\int_{\Om} b(x)\phi_{1}^{\ba}dx >0$. Then $\ds\inf_{u\in\mc N_{\la}^{-}} J_{\la}(u)=0$ for every $\la>\la_1$.
\end{Corollary}

\begin{proof}
We know that for $\la>\la_1$, $\phi_1\in E^{-}_{\la}$, so $E_{\la}^{-}\cap B^{+}\not= \emptyset$. Hence by Theorem \ref{ttt2}, $\ds\inf_{u\in\mc N_{\la}^{-}} J_{\la}(u)=0$.\QED
 \end{proof}
 \noi {\bf Proof of Theorem \ref{sp3}:} Corollary \ref{c2} completes the proof of Theorem \ref{sp3}.\\

\noi The next result follow the similar result without any assumption but with the large $\la$.
\begin{Lemma}
 There exists $\tilde{\la}$ such that $\ds\inf_{u\in\mc N_{\la}^{-}} J_{\la}(u)=0$ for every $\la>\tilde{\la}$.
\end{Lemma}
\begin{proof}
Let $ u\in X_0$ such that $\int_{\Om} b(x)|u|^{\ba}dx >0$. Then choose $\tilde{\la}$ sufficiently large so that $\|u\|^p -\la \int_{\Om} |u|^{p} dx <0$
whenever $\la>\tilde{\la}$. Thus for $\la>\tilde{\la}$, $u\in E_{\la}^{-}\cap B^{+}$ and hence the result follows from Theorem \ref{ttt2}.\QED
\end{proof}

\noi Finally, we show that $J_\la$ is unbounded below on $\mc N_{\la}^{+}$ where $\la$ is sufficiently large.

 \begin{Theorem}
If $E_{\la}^{-}\cap B^{0}\not= \emptyset$, then $J_{\la}(u)$ is unbounded on $\mc N_{\la}^{+}$.
\end{Theorem}

 \begin{proof}
Let $u\in E_{\la}^{-}\cap B^{0}$. Then by decreasing $u$ slightly in $\{x\in \Om: b(x)>0\}$, for given $\e>0$, we can find $v\in X_0$ with $\|v\|=1$ such that
$\|u-v\|<\e$, $ -\e<\int_{\Om} b(x)|v|^{\ba} dx<0$ and $\|v\|^p -\la \int_{\Om} |v|^{p} dx<\frac{1}{2}(\|u\|^p -\la \int_{\Om} |u|^{p} dx).$
 Therefore there exist $\de>0$ and a sequence $\{v_k\}\in E_{\la}^{-}\cap B^{-}$ such that
$\|v_k\|^p  -\la \int_{\Om} |v_k|^{p} dx<-\de \;\mbox{and}\;\ds\lim_{k\ra \infty} \int_{\Om} b(x)|v_k|^{\ba}dx \ra 0.$
Hence \[\lim_{k\ra \infty} t(v_k)=\ds\lim_{k\ra\infty} \left[\frac{\|v_k\|^{p} - \la \int_{\Om}
|v_k|^{p} dx}{\int_{\Om} b(x)|v_k|^{\ba} dx}\right]^{\frac{1}{\ba-p}}= \infty.\]
Now $t(v_k)v_k\in \mc N_{\la}^{+}$, we have
\[J_{\la}(t(v_k)v_k)= \left(\frac{1}{p}-\frac{1}{\ba}\right)|t(v_k)|^{p} \left[\|v_k\|^p- \la\int_{\Om} |v_k|^{p} dx\right] \leq \left(\frac{1}{p}-\frac{1}{\ba}\right)|t(v_k)|^{p} (-\de)\ra -\infty \]
as $k\ra\infty$ and so $J_{\la}(u)$ is not bounded below on $\mc N_{\la}^{+}$. \QED
\end{proof}
\begin{Corollary}
 There exists $\hat{\la}$ such that $J_{\la}(u)$ is unbounded below on $\mc N_{\la}^{+}$ whenever $\la>\hat{\la}$.
\end{Corollary}
\begin{proof}
Let $ u\in X_0$ with $\|u\|=1$ and $\int_{\Om} b(x)|u|^{\ba}dx =0$. Choose $\hat{\la}$ sufficiently large so that $\|u\|^p -\la \int_{\Om} |u|^{p} dx <0$
whenever $\la>\hat{\la}$. Thus for $\la>\hat{\la}$, $u\in E_{\la}^{-}\cap B^{0}$ and hence the proof is complete.\QED
\end{proof}
{\bf Acknowledgements:} Authors thank the reviewer for his/her thorough review and highly appreciate the comments and
suggestions, which significantly contributed to improving this article.


\end{document}